\documentclass[onefignum,onetabnum]{siamart251216}

\usepackage{amsfonts}
\usepackage{amsmath}
\usepackage{amssymb}
\usepackage{mathtools}
\usepackage{algorithm}
\usepackage{algpseudocode}

\headers
{Adaptive Multiresolution Diffusion Operators}
{C. Tantardini et al.}

\title{
Adaptive Multiresolution Diffusion Operators:
A Variational Theory on Evolving Multiresolution Spaces
}

\author{
Christian Tantardini
\thanks{Center for Integrative Petroleum Research,
King Fahd University of Petroleum and Minerals,
Dhahran 31261, Kingdom of Saudi Arabia
(\email{christiantantardini@ymail.com}).
}
\and
Stig Rune Jensen
\thanks{
Hylleraas Centre for Quantum Molecular Sciences,
Department of Chemistry,
UiT The Arctic University of Norway,
N-9037 Tromsø, Norway
(\email{stig.r.jensen@uit.no}).
}
\and
Roberto Di Remigio Eik{\AA}s
\thanks{Algorithmiq S.r.l., c/o STLex, Via della Chiusa 15, 20123, Milano, Italia
(\email{roberto@algorithmiq.tech})
}
\and
Joakim Henrik Beck
\thanks{
Center for Integrative Petroleum Research,
King Fahd University of Petroleum and Minerals,
Dhahran 31261, Kingdom of Saudi Arabia.
(\email{joakim.beck@kfupm.edu.sa}).
}
}

\usepackage{hyperref}
\usepackage{cleveref}

\begin{document}

\maketitle

\begin{abstract}
We develop a variational framework for state-dependent diffusion on adaptive multiresolution representations in which the diffusion operator is generated by the adaptive state itself. The state consists of an admissible multiresolution tree, its active approximation space and basis, and the corresponding coefficient representation. It determines a symmetric nonnegative interaction form and an associated positive semidefinite intrinsic diffusion operator. In contrast to classical adaptive wavelet methods, where a prescribed operator is represented on an evolving approximation space, refinement and coarsening here modify simultaneously the representation, interaction graph, and operator.
Because the adaptive hierarchy evolves through discrete topological changes, the coupled dynamics are formulated through a time-discrete variational principle rather than a differential evolution on a fixed space. We establish existence of the discrete updates, a discrete energy inequality, the coefficient-space null mode, and contractivity for frozen adaptive states. For regularized inverse problems, the construction yields Adaptive Multiresolution Diffusion Imaging (AMDI), combining data fidelity, intrinsic diffusion, coefficient sparsity, and tree complexity in a state-dependent energy.
Numerical experiments verify the assembled operator identities, examine refinement-commutator decay, and confirm discrete energy dissipation. Adaptive Haar and higher-order multiwavelet calculations demonstrate localization of resolution on heterogeneous data. In denoising, AMDI retains high structural reconstruction quality with less than 10\% of the full active representation, with stable behavior across held-out noise realizations.
\end{abstract}



\section{Introduction}
\label{sec:introduction}

Adaptive numerical methods provide one of the most successful frameworks for
the numerical solution of partial differential equations, integral equations,
and variational problems. Their central objective is to distribute
computational effort only where it is required, thereby achieving high
accuracy with substantially fewer degrees of freedom than uniform
discretizations. Over the past decades this philosophy has led to rigorous
theories of adaptive mesh refinement, adaptive finite elements, nonlinear
approximation, and adaptive wavelet methods, establishing optimal convergence,
error control, and computational complexity for broad classes of operator
equations
\cite{Brandt1977,
BergerOliger1984,
Dorfler1996,
MorinNochettoSiebert2002,
BinevDahmenDeVore2004,
DeVore1998,
Temlyakov2011,
Dahmen1997,
Dahmen2001JCAM,
Cohen2003Book}.

Among these developments, adaptive wavelet methods occupy a distinguished
position because they combine hierarchical approximation, sparse
representations, and rigorous convergence theory within a unified framework.
Wavelet representations yield near-optimal algorithms for broad classes of
linear and nonlinear operator equations while providing mathematically
controlled refinement and coarsening
\cite{Beylkin1992Operators,
BeylkinCoifmanRokhlin1991,
Alpert1993,
AlpertBeylkinCoifmanRokhlin1993,
BeylkinKeiser1997,
CohenDahmenDeVore2001,
CohenDahmenDeVore2002,
CohenDahmenDeVore2005,
Stevenson2003,
Stevenson2004Compressibility,
DahlkeDahmenUrban2002,
DahmenKunoth1992}.
In these methods, the adaptive approximation space is constructed for a
differential, integral, or variational operator that is specified independently
of the adaptive hierarchy.

A complementary direction concerns the efficient representation of the
operators themselves. Wavelet compression, nonstandard operator forms,
hierarchical matrices, tensor representations, multiresolution operator
calculus, and fast summation methods have shown that dense and nonlocal
operators can often be approximated with substantially reduced computational
complexity
\cite{Beylkin1992Operators,
BeylkinMohlenkamp2002,
BeylkinMohlenkamp2005SIAM,
HackbuschBorm2002,
Hackbusch2017,
GreengardRokhlin1987,
BeylkinMonzon2002,
BeylkinMonzon2005}.
These ideas have led to adaptive multiwavelet computational environments in
which functions and operators are represented and evaluated directly at a
prescribed numerical accuracy
\cite{HarrisonEtAl2004,
FannEtAl2004,
HarrisonEtAlMADNESS2008,
FredianiFossgaardFlaRuud2013,
WindEtAl2023,
BjorgveEtAl2024,
DinvayZabelinaFrediani2024}.

Diffusion operators also play a central role in multiscale data analysis.
Diffusion maps, diffusion wavelets, spectral graph methods, and manifold
learning use the spectral properties of a prescribed diffusion or graph
Laplacian to construct coordinates and multiresolution bases adapted to the
geometry of complex datasets
\cite{CoifmanLafon2006,
CoifmanMaggioni2006,
CoifmanEtAl2005,
BelkinNiyogi2003,
BelkinNiyogi2008,
Singer2006,
SingerWu2012,
Chung1997}.
In these constructions, however, the diffusion operator is defined before the
multiresolution representation and subsequently determines the basis.

State-dependent diffusion and graph regularization are also well established
in image processing.  Nonlinear anisotropic diffusion, nonlocal filtering,
nonlocal variational models, and graph-based $p$-Laplacian regularization
construct diffusivities or interaction weights from the current signal or
image itself
\cite{PeronaMalik1990,
BuadesCollMorel2005,
GilboaOsher2008,
ElmoatazLezorayBougleux2008}.
In these approaches, however, the underlying set of degrees of freedom on
which the operator acts is fixed independently of the state-dependent
weights.  The distinction pursued here is different: the active
multiresolution index set, the approximation space, the interaction graph,
and the interaction weights are all components of the evolving adaptive
state.  Refinement or coarsening therefore changes not only the coefficients
or edge weights, but also the finite-dimensional space and graph on which the
intrinsic operator is defined.

More recently, operator-learning methods such as DeepONet, Fourier Neural
Operators, Graph Kernel Networks, Multiwavelet Neural Operators, and Wavelet
Neural Operators have demonstrated that nonlinear mappings between
infinite-dimensional function spaces can be approximated directly from data
\cite{LuEtAl2021DeepONet,
LiEtAl2021FNO,
LiKovachkiAzizzadenesheli2020,
GuptaXiaoBogdan2021,
TripuraChakraborty2023,
KovachkiEtAl2023}.
Although these approaches differ substantially in their mathematical
formulation and numerical implementation, the operator or the class of
operators to be learned is defined independently of the adaptive
multiresolution state.

The existing directions may therefore be summarized schematically as
\begin{align}
\text{fixed operator}
&\longrightarrow
\text{adaptive approximation},
\\
\text{fixed operator}
&\longrightarrow
\text{operator compression},
\\
\text{fixed diffusion operator}
&\longrightarrow
\text{multiresolution basis},
\\
\text{prescribed operator family}
&\longrightarrow
\text{learned operator surrogate}.
\end{align}
These constructions cover fixed-operator approximation, operator compression,
operator-induced multiresolution bases, and learned surrogates.  State-dependent
diffusion methods additionally allow the interaction weights to depend on the
current signal.  What remains fixed in these settings is the representation
on which the operator acts.  The present construction instead couples the
active multiresolution index set, its approximation space, the interaction
topology, the state-dependent weights, and the represented coefficients.

The central idea of the present work is therefore to make the adaptive
multiresolution representation itself part of the operator state. 
This reversal is not merely a change of perspective. Once the operator is
generated by the adaptive state itself, the governing dynamics can no longer
be separated from the evolving approximation space. The operator becomes an
intrinsic component of the state, and every refinement or coarsening event
changes both the representation and the diffusion process simultaneously.
Rather than
prescribing the diffusion operator in advance, we derive it directly from the
adaptive multiresolution representation. The adaptive hierarchy is therefore
not merely a numerical discretization of a previously defined model but becomes
part of the mathematical definition of the dynamics. This leads to the
construction
\begin{align}
\text{adaptive multiresolution state}
\longrightarrow
\text{intrinsic diffusion operator},
\label{eq:reverse_direction}
\end{align}
in which the operator is generated by the geometry of the active adaptive tree,
the multiwavelet basis, and the evolving expansion coefficients.

The mathematical foundation of the proposed framework is the adaptive
multiresolution state
\begin{align}
\mathfrak V
=
\left(
\Lambda,
V_{\Lambda},
\Phi_{\Lambda},
\mathcal T_{\Lambda},
\mathbf c
\right),
\label{eq:adaptive_state_intro}
\end{align}
where $\Lambda$ is an admissible adaptive index tree,
$V_{\Lambda}$ is the associated approximation space,
$\Phi_{\Lambda}$ is the active multiwavelet basis,
$\mathcal T_{\Lambda}$ describes the hierarchical geometry, and
$\mathbf c$ denotes the vector of expansion coefficients. From this state, we
construct a symmetric positive semidefinite interaction form and the associated
intrinsic diffusion operator. Because both the adaptive tree and the
coefficients evolve, the operator is itself a state-dependent dynamical
quantity.

Consequently, the adaptive state becomes the primary mathematical variable of
the theory, while the diffusion operator is regarded as a derived quantity
rather than an independent model input.

Because the adaptive hierarchy changes through discrete topological events,
its evolution cannot naturally be described by a classical differential
equation on a fixed function space.  We therefore formulate the coupled
evolution directly through a time-discrete variational scheme in which the
adaptive tree and the represented function are updated by minimizing a common
incremental functional.  The construction is motivated by minimizing-movement
and variational-evolution ideas
\cite{Ambrosio1995,
GianazzaGobbinoSavare1994,
AmbrosioGigliSavare2008,
Braides2002,
Braides2014,
MielkeOrtiz2008},
but the present work concerns the discrete adaptive evolution itself.  In
particular, no continuous-time limit as the time step tends to zero is assumed
or established here.

The principal contribution of this work is therefore a variational framework
in which adaptive multiresolution geometry, sparse approximation, and the
governing diffusion operator are treated as a single evolving mathematical
object. Rather than approximating a prescribed operator on an adaptive space,
the operator itself is generated by the adaptive state and evolves together
with it. This removes the classical distinction between the governing operator and its adaptive discretization and leads naturally to a coupled variational evolution of the coefficients and the adaptive hierarchy. Within this framework, we establish
the basic structural properties of the intrinsic operator, formulate the
time-discrete variational dynamics, and extend the construction to
inverse problems and adaptive multiresolution diffusion imaging.

The remainder of the paper is organized as follows. 
Section~\ref{sec:spaces} introduces the
ambient Hilbert space, admissible adaptive multiresolution states, and the
associated function spaces. Section~\ref{sec:intrinsic_operator} constructs
the intrinsic adaptive diffusion operator, while
Sections~\ref{sec:variational} and~\ref{sec:discrete_evolution} develop its
variational formulation and the coupled variational evolution.
Section~\ref{sec:analysis} establishes the principal analytical properties,
including existence, stability, energy dissipation, and contractivity.
Section~\ref{sec:inverse_problems} extends the framework to inverse problems
and adaptive multiresolution diffusion imaging, and
Section~\ref{sec:algorithm} presents the corresponding adaptive numerical
algorithm. Finally,
Section~\ref{sec:numerical_validation} outlines the numerical validation
strategy, and
Section~\ref{sec:conclusions} summarizes the main results and discusses future
research directions.

\section{Adaptive Multiresolution Function Spaces}
\label{sec:spaces}
 
We first specify the fixed ambient Hilbert space and the multiresolution
structure from which the admissible adaptive states are constructed. Although
the active approximation spaces have tree-dependent finite dimensions, they
are all realized as subspaces of the same ambient Hilbert space. This provides
a common setting in which states associated with different adaptive trees can
be compared and transferred.

\subsection{Ambient Hilbert space and multiresolution hierarchy}

Let $\Omega\subset\mathbb R^{d}$ be a bounded Lipschitz domain and let
\begin{align}
H
=
L^{2}(\Omega)
\label{eq:ambient_H}
\end{align}
be the fixed ambient Hilbert space, equipped with inner product
$\langle\cdot,\cdot\rangle_{H}$ and norm $\|\cdot\|_{H}$.
All adaptive approximation spaces introduced below are finite-dimensional
subspaces of $H$.

Let
\begin{align}
V_{0}
\subset
V_{1}
\subset
V_{2}
\subset
\cdots
\subset
H
\label{eq:nested_spaces}
\end{align}
be a nested multiresolution hierarchy satisfying
\begin{align}
\overline{
\bigcup_{\ell\geq 0}V_{\ell}
}^{\,H}
=
H.
\label{eq:density_hierarchy}
\end{align}
For each level $\ell\geq0$, let $W_{\ell}$ denote a detail space such that
\begin{align}
V_{\ell+1}
=
V_{\ell}
\oplus
W_{\ell},
\label{eq:detail_space}
\end{align}
where $\oplus$ denotes a direct sum. In the orthogonal case,
\begin{align}
W_{\ell}
=
V_{\ell+1}\ominus V_{\ell}.
\label{eq:orthogonal_detail_space}
\end{align}

Let $\mathcal I$ be the countable index set associated with the resulting
multiresolution dictionary
\begin{align}
\Psi
=
\left\{
\psi_{\lambda}
:
\lambda\in\mathcal I
\right\}
\subset H.
\label{eq:ambient_dictionary}
\end{align}
The subset $\mathcal I_{0}\subset\mathcal I$ indexes the coarsest-level
functions, while the remaining indices are associated with the detail spaces
$W_{\ell}$. For a detail function, we write
\begin{align}
\lambda
=
\left(
\ell,
k,
q
\right),
\label{eq:multi_index}
\end{align}
where $\ell$ denotes the scale level, $k$ the spatial index, and $q$ the
local polynomial or multiwavelet component.

We assume that the functions indexed by $\mathcal I_{0}$ span $V_{0}$ and
that the remaining basis functions are organized according to the detail
spaces $W_{\ell}$. Thus the multiresolution hierarchy and the dictionary
$\Psi$ describe the same underlying decomposition of $H$.

We assume that $\Psi$ is a Riesz basis of $H$. Hence there exist constants
$0<A_{\Psi}\leq B_{\Psi}<\infty$ such that, for every
$\mathbf c=(c_{\lambda})_{\lambda\in\mathcal I}\in\ell^{2}(\mathcal I)$,
\begin{align}
A_{\Psi}
\|\mathbf c\|_{\ell^{2}(\mathcal I)}^{2}
\leq
\left\|
\sum_{\lambda\in\mathcal I}
c_{\lambda}\psi_{\lambda}
\right\|_{H}^{2}
\leq
B_{\Psi}
\|\mathbf c\|_{\ell^{2}(\mathcal I)}^{2}.
\label{eq:Riesz_bounds}
\end{align}

Because the state-dependent interaction weights introduced below compare
coefficient amplitudes associated with different indices, the normalization
of the multiresolution dictionary is regarded as part of the representation.
Throughout the paper, we use the fixed normalization
\begin{align}
\|\psi_{\lambda}\|_{H}
=
1,
\qquad
\lambda\in\mathcal I.
\label{eq:basis_H_normalization}
\end{align}
This entails no loss of the Riesz-basis property.  Indeed,
\eqref{eq:Riesz_bounds} implies uniform upper and lower bounds for the
individual basis-function norms, so diagonal normalization defines a bounded
and boundedly invertible rescaling of the coefficient space.  With this
normalization fixed, coefficient differences appearing in the interaction
kernels are unambiguous; scale dependence of the interaction itself is
represented explicitly through $\kappa_{\ell}$.
The series converges in $H$. In particular, every finite subset of $\Psi$
is linearly independent, and its coefficient representation is unique.

\subsection{Admissible adaptive trees}

We assume that the coarsest-level index set $\mathcal I_{0}$ is finite.
The refinement hierarchy induces a parent map
\begin{align}
\operatorname{par}
:
\mathcal I\setminus\mathcal I_{0}
\longrightarrow
\mathcal I,
\label{eq:parent_map}
\end{align}
which associates each non-root index with its parent in the multiresolution
hierarchy.

A finite set $\Lambda\subset\mathcal I$ is called
\emph{tree admissible} if
\begin{align}
\mathcal I_{0}
\subseteq
\Lambda
\label{eq:root_inclusion}
\end{align}
and
\begin{align}
\lambda\in\Lambda\setminus\mathcal I_{0}
\quad
\Longrightarrow
\quad
\operatorname{par}(\lambda)\in\Lambda.
\label{eq:tree_admissibility}
\end{align}
Thus, every active detail function is accompanied by its complete ancestral
path to the coarsest level.

The class of finite admissible trees is
\begin{align}
\mathbb T_{\mathrm{adm}}
=
\left\{
\Lambda\subset\mathcal I
:
|\Lambda|<\infty,\;
\mathcal I_{0}\subseteq\Lambda,\;
\Lambda\ \text{satisfies \eqref{eq:tree_admissibility}}
\right\}.
\label{eq:admissible_tree_class}
\end{align}

Each $\Lambda\in\mathbb T_{\mathrm{adm}}$ inherits the parent--child
structure of the ambient multiresolution hierarchy. We denote the resulting
finite rooted graph by
\begin{align}
\mathcal T_{\Lambda}
=
\left(
\Lambda,
\mathcal E_{\Lambda}^{\mathrm{par}}
\right),
\label{eq:adaptive_tree_graph}
\end{align}
where
\begin{align}
\mathcal E_{\Lambda}^{\mathrm{par}}
=
\left\{
\left(
\operatorname{par}(\lambda),
\lambda
\right)
:
\lambda\in\Lambda\setminus\mathcal I_{0}
\right\}.
\label{eq:adaptive_tree_edges}
\end{align}
If $\mathcal I_{0}$ contains more than one coarsest-level function,
$\mathcal T_{\Lambda}$ is formally a rooted forest; throughout the paper we
use the term adaptive tree for this structure.

For each $\Lambda\in\mathbb T_{\mathrm{adm}}$, define
\begin{align}
\Phi_{\Lambda}
&=
\left\{
\psi_{\lambda}
\right\}_{\lambda\in\Lambda},
\label{eq:active_basis}
\\
V_{\Lambda}
&=
\operatorname{span}\Phi_{\Lambda}
\subset H,
\label{eq:active_space}
\\
N_{\Lambda}
&=
\dim V_{\Lambda}
=
|\Lambda|.
\label{eq:active_dimension}
\end{align}

Every $u_{\Lambda}\in V_{\Lambda}$ has the representation
\begin{align}
u_{\Lambda}
=
\sum_{\lambda\in\Lambda}
c_{\lambda}\psi_{\lambda}.
\label{eq:adaptive_expansion}
\end{align}

By the Riesz-basis property, the representation
\eqref{eq:adaptive_expansion} is unique and satisfies
\begin{align}
A_{\Psi}
\|\mathbf c_{\Lambda}\|_{\ell^{2}(\Lambda)}^{2}
\leq
\|u_{\Lambda}\|_{H}^{2}
\leq
B_{\Psi}
\|\mathbf c_{\Lambda}\|_{\ell^{2}(\Lambda)}^{2}.
\label{eq:active_Riesz_bounds}
\end{align}
where
\begin{align}
\mathbf c_{\Lambda}
=
\left(
c_{\lambda}
\right)_{\lambda\in\Lambda}
\in
\mathbb R^{N_{\Lambda}}.
\label{eq:coefficient_vector}
\end{align}

\subsection{Adaptive multiresolution state}

\begin{definition}[Adaptive multiresolution state]
\label{def:adaptive_state}
For $\Lambda\in\mathbb T_{\mathrm{adm}}$ and
$\mathbf c_{\Lambda}\in\mathbb R^{N_{\Lambda}}$, the associated adaptive
multiresolution state is
\begin{align}
\mathfrak V
\left(
\Lambda,
\mathbf c_{\Lambda}
\right)
=
\left(
\Lambda,
V_{\Lambda},
\Phi_{\Lambda},
\mathcal T_{\Lambda},
\mathbf c_{\Lambda}
\right).
\label{eq:adaptive_state}
\end{align}
The function represented by the state is
\begin{align}
u_{\mathfrak V}
=
\sum_{\lambda\in\Lambda}
c_{\lambda}\psi_{\lambda}
\in
V_{\Lambda}.
\label{eq:state_function}
\end{align}
\end{definition}

Since $V_{\Lambda}$, $\Phi_{\Lambda}$, and $\mathcal T_{\Lambda}$ are
determined uniquely by $\Lambda$, the independent variables of an adaptive
state are the admissible tree $\Lambda$ and its coefficient vector
$\mathbf c_{\Lambda}$. We therefore define the adaptive state space by
\begin{align}
\mathbb A
=
\left\{
\mathfrak V
\left(
\Lambda,
\mathbf c_{\Lambda}
\right)
:
\Lambda\in\mathbb T_{\mathrm{adm}},
\;
\mathbf c_{\Lambda}\in\mathbb R^{N_{\Lambda}}
\right\}.
\label{eq:adaptive_state_set}
\end{align}
The state space $\mathbb A$ is not a linear space; rather, it is a union of
finite-dimensional state sets indexed by the discrete tree configurations
$\Lambda\in\mathbb T_{\mathrm{adm}}$.
Nevertheless, every represented function $u_{\mathfrak V}$ belongs to the
common ambient Hilbert space $H$. This fixed ambient realization allows states
associated with different adaptive trees to be compared without identifying
their coefficient spaces.

\subsection{Transfer between adaptive states}

To compare states associated with different adaptive trees, we introduce
uniformly stable projections onto the active approximation spaces. For each
$\Gamma\in\mathbb T_{\mathrm{adm}}$, let
\begin{align}
P_{\Gamma}
:
H
\longrightarrow
V_{\Gamma}
\label{eq:projection_Gamma}
\end{align}
be a bounded linear projector satisfying the reproduction property
\begin{align}
P_{\Gamma}v
=
v,
\qquad
v\in V_{\Gamma}.
\label{eq:projection_reproduction}
\end{align}
We assume that these projectors are uniformly stable, namely,
\begin{align}
\sup_{\Gamma\in\mathbb T_{\mathrm{adm}}}
\left\|
P_{\Gamma}
\right\|_{H\rightarrow H}
\leq
C_{Q}
<
\infty.
\label{eq:projection_uniform_stability}
\end{align}
The $H$-orthogonal projection is the canonical example, for which
$C_{Q}=1$.

For $\Lambda,\Gamma\in\mathbb T_{\mathrm{adm}}$, the transfer operator from
$V_{\Lambda}$ to $V_{\Gamma}$ is defined by
\begin{align}
\mathcal Q_{\Gamma\leftarrow\Lambda}
=
P_{\Gamma}\big|_{V_{\Lambda}}
:
V_{\Lambda}
\longrightarrow
V_{\Gamma}.
\label{eq:transfer_operator}
\end{align}
The uniform projector bound immediately yields
\begin{align}
\left\|
\mathcal Q_{\Gamma\leftarrow\Lambda}v
\right\|_{H}
\leq
C_{Q}
\|v\|_{H},
\qquad
v\in V_{\Lambda},
\label{eq:transfer_stability}
\end{align}
with a constant independent of the admissible trees.

If $\Lambda\subseteq\Gamma$, then
$V_{\Lambda}\subseteq V_{\Gamma}$ and the reproduction property
\eqref{eq:projection_reproduction} gives
\begin{align}
\mathcal Q_{\Gamma\leftarrow\Lambda}v
=
v,
\qquad
v\in V_{\Lambda}.
\label{eq:nested_injection}
\end{align}
Hence nested refinement introduces no transfer error, whereas transfer
between non-nested adaptive spaces remains uniformly stable in the ambient
$H$ norm.

\section{Intrinsic Adaptive Multiresolution Diffusion Operators}
\label{sec:intrinsic_operator}

Throughout this section, let
$\mathfrak V=\mathfrak V(\Lambda,\mathbf c_{\Lambda})\in\mathbb A$
be an admissible adaptive state in the sense of
Definition~\ref{def:adaptive_state}. The interaction weights associated with
$\mathfrak V$ are evaluated at this state and are held fixed when the
corresponding operator acts on a coefficient vector.
Thus, for each fixed state the operator defined below is linear. For a fixed
tree $\Lambda$, however, the resulting operator family may depend nonlinearly
on the coefficient component $\mathbf c_{\Lambda}$ of the adaptive state.

\subsection{Adaptive interaction geometry}

For a fixed admissible tree $\Lambda$, let
\begin{align}
\mathcal E_{\Lambda}
\subset
\left\{
(\lambda,\mu)\in\Lambda\times\Lambda
:
\lambda\neq\mu
\right\}
\label{eq:interaction_edges}
\end{align}
be a symmetric interaction relation.
An edge may connect basis functions
that are related by one or more of the following:

\begin{enumerate}
\item parent--child adjacency;
\item spatial overlap;
\item nearest-neighbour support;
\item cross-scale interaction;
\item common-parent or sibling interaction.
\end{enumerate}

We denote the indicator of the interaction relation by
\begin{align}
\chi_{\lambda\mu}
=
\begin{cases}
1,
&
(\lambda,\mu)\in\mathcal E_{\Lambda},
\\
0,
&
(\lambda,\mu)\notin\mathcal E_{\Lambda}.
\end{cases}
\label{eq:interaction_indicator}
\end{align}
Symmetry of $\mathcal E_{\Lambda}$ implies
$\chi_{\lambda\mu}=\chi_{\mu\lambda}$.

For $\lambda,\mu\in\Lambda$, define the symmetric weight
\begin{align}
\omega_{\lambda\mu}
\left(
\mathfrak V
\right)
=
\omega_{\mu\lambda}
\left(
\mathfrak V
\right)
\geq 0,
\label{eq:symmetric_weights}
\end{align}
with
\begin{align}
\omega_{\lambda\lambda}
\left(
\mathfrak V
\right)
=
0.
\label{eq:zero_self_weight}
\end{align}

When the weights depend on the coefficient component of the state, we write
$c_{\lambda}^{\mathfrak V}$ for the $\lambda$th component of
$\mathbf c_{\Lambda}$ in \eqref{eq:adaptive_state}.

A state-dependent construction used in the numerical realization is
\begin{align}
\omega_{\lambda\mu}
\left(
\mathfrak V
\right)
=
\chi_{\lambda\mu}
\,
\kappa_{\mathrm x}
\left(
d_{\mathrm x}(\lambda,\mu)
\right)
\,
\kappa_{\ell}
\left(
|\ell_{\lambda}-\ell_{\mu}|
\right)
\,
\kappa_{\mathrm c}
\left(
|c_{\lambda}^{\mathfrak V}-c_{\mu}^{\mathfrak V}|
\right)
\,
\kappa_{\mathrm r}
\left(
\lambda,\mu
\right),
\label{eq:weight_construction}
\end{align}
where the refinement-decay factor is
\begin{align}
\kappa_{\mathrm r}
\left(
\lambda,\mu
\right)
=
\begin{cases}
2^{-\gamma_{\mathrm r}
\max\{\ell_{\lambda},\ell_{\mu}\}},
&
\mu=\operatorname{par}(\lambda)
\ \text{or}\
\lambda=\operatorname{par}(\mu),
\\[1mm]
1,
&
\text{otherwise},
\end{cases}
\qquad
\gamma_{\mathrm r}\geq0.
\label{eq:refinement_decay_kernel}
\end{align}
The factor $\kappa_{\mathrm r}$ suppresses newly introduced cross-level
parent--child couplings as the refinement level increases.  Setting
$\gamma_{\mathrm r}=0$ recovers the construction without absolute-level
decay.

For coefficient-sensitive diffusion, one may choose, for $\sigma>0$,
\begin{align}
\kappa_{\mathrm c}(r)
=
\frac{1}{1+(r/\sigma)^{2}}
\label{eq:coefficient_kernel}
\end{align}
or
\begin{align}
\kappa_{\mathrm c}(r)
=
\exp
\left(
-\left(\frac{r}{\sigma}\right)^{2}
\right).
\label{eq:gaussian_coefficient_kernel}
\end{align}

For every fixed state $\mathfrak V$, the weights
$\omega_{\lambda\mu}(\mathfrak V)$ are fixed and the resulting
coefficient-space operator is linear. For a fixed tree $\Lambda$, the
dependence of $L_{\mathfrak V}$ on the coefficient component
$\mathbf c_{\Lambda}$ may nevertheless be nonlinear through the weights.
Thus the state dependence of the operator must be distinguished from its
linear action at a fixed adaptive state. When the weights depend only on the
tree and basis geometry, the coefficient dependence is absent.

\subsection{Intrinsic Dirichlet form}

\begin{definition}[Adaptive interaction form]
\label{def:adaptive_form}
For a fixed adaptive state $\mathfrak V$ and coefficient vectors
$\mathbf x,\mathbf y\in\mathbb R^{N_{\Lambda}}$, define
\begin{align}
a_{\mathfrak V}
\left(
\mathbf x,
\mathbf y
\right)
=
\frac{1}{2}
\sum_{\lambda,\mu\in\Lambda}
\omega_{\lambda\mu}
\left(
\mathfrak V
\right)
\left(
x_{\lambda}-x_{\mu}
\right)
\left(
y_{\lambda}-y_{\mu}
\right).
\label{eq:adaptive_bilinear_form}
\end{align}
\end{definition}

For fixed $\mathfrak V$, the weights are fixed and
$a_{\mathfrak V}$ is therefore a symmetric bilinear form on
$\mathbb R^{N_{\Lambda}}$. Its associated interaction energy is
\begin{align}
\mathcal D_{\mathfrak V}
\left(
\mathbf x
\right)
=
\frac{1}{2}
a_{\mathfrak V}
\left(
\mathbf x,
\mathbf x
\right)
=
\frac{1}{4}
\sum_{\lambda,\mu\in\Lambda}
\omega_{\lambda\mu}
\left(
\mathfrak V
\right)
\left|
x_{\lambda}-x_{\mu}
\right|^{2}.
\label{eq:intrinsic_Dirichlet_energy}
\end{align}
The nonnegativity of the weights implies
\begin{align}
\mathcal D_{\mathfrak V}
\left(
\mathbf x
\right)
\geq
0.
\label{eq:Dirichlet_nonnegative}
\end{align}

The interaction energy of the adaptive state itself is obtained by evaluating
\eqref{eq:intrinsic_Dirichlet_energy} at its coefficient vector
$\mathbf c_{\Lambda}$.

\subsection{Intrinsic diffusion operator}

\begin{definition}[Intrinsic coefficient-space diffusion operator]
\label{def:intrinsic_operator}
For a fixed adaptive state $\mathfrak V$, define
\begin{align}
L_{\mathfrak V}
:
\mathbb R^{N_{\Lambda}}
\longrightarrow
\mathbb R^{N_{\Lambda}}
\label{eq:coefficient_operator_mapping}
\end{align}
componentwise by
\begin{align}
\left(
L_{\mathfrak V}\mathbf x
\right)_{\lambda}
=
\sum_{\mu\in\Lambda}
\omega_{\lambda\mu}
\left(
\mathfrak V
\right)
\left(
x_{\lambda}-x_{\mu}
\right).
\label{eq:coefficient_operator}
\end{align}
\end{definition}

Let
\begin{align}
W_{\mathfrak V}
&=
\left[
\omega_{\lambda\mu}
\left(
\mathfrak V
\right)
\right]_{\lambda,\mu\in\Lambda},
\label{eq:weight_matrix}
\\
D_{\mathfrak V}
&=
\operatorname{diag}
\left(
\sum_{\mu\in\Lambda}
\omega_{\lambda\mu}
\left(
\mathfrak V
\right)
\right)_{\lambda\in\Lambda}.
\label{eq:degree_matrix}
\end{align}
Then
\begin{align}
L_{\mathfrak V}
=
D_{\mathfrak V}
-
W_{\mathfrak V},
\label{eq:matrix_L}
\end{align}
so $L_{\mathfrak V}$ is the weighted graph Laplacian associated with the
adaptive interaction geometry.

To obtain the corresponding operator on $V_{\Lambda}\subset H$, introduce
the Gram matrix
\begin{align}
G_{\Lambda}
=
\left[
\left\langle
\psi_{\lambda},
\psi_{\mu}
\right\rangle_{H}
\right]_{\lambda,\mu\in\Lambda}.
\label{eq:Gram_matrix}
\end{align}
By the Riesz-basis bounds in \eqref{eq:Riesz_bounds},
$G_{\Lambda}$ is symmetric positive definite, with spectral bounds uniform
over the admissible trees.

For $u_{\Lambda}\in V_{\Lambda}$, let
$\mathbf x_{\Lambda}\in\mathbb R^{N_{\Lambda}}$ denote its coefficient
vector with respect to $\Phi_{\Lambda}$. We define the intrinsic
function-space diffusion operator by
\begin{align}
\mathcal L_{\mathfrak V}u_{\Lambda}
=
\sum_{\lambda\in\Lambda}
\left(
G_{\Lambda}^{-1}
L_{\mathfrak V}
\mathbf x_{\Lambda}
\right)_{\lambda}
\psi_{\lambda}.
\label{eq:function_operator}
\end{align}
Equivalently, for every $v_{\Lambda}\in V_{\Lambda}$ with coefficient vector
$\mathbf y_{\Lambda}$,
\begin{align}
\left\langle
\mathcal L_{\mathfrak V}u_{\Lambda},
v_{\Lambda}
\right\rangle_{H}
=
a_{\mathfrak V}
\left(
\mathbf x_{\Lambda},
\mathbf y_{\Lambda}
\right).
\label{eq:function_operator_variational}
\end{align}

Thus, the intrinsic operator construction defines the state-indexed
assignment
\begin{align}
\mathfrak V
=
\mathfrak V(\Lambda,\mathbf c_{\Lambda})
\longmapsto
\mathcal L_{\mathfrak V}
\in
\mathcal B(V_{\Lambda}),
\label{eq:state_to_operator}
\end{align}
where $\mathcal B(V_{\Lambda})$ denotes the space of bounded linear operators
on $V_{\Lambda}$. Since the active space $V_{\Lambda}$ depends on the tree
component of the adaptive state, the operators $\mathcal L_{\mathfrak V}$
do not belong, in general, to a single fixed operator space. Equation
\eqref{eq:state_to_operator} should therefore be understood as an
operator family indexed by the adaptive state. For every fixed
$\mathfrak V$, the operator $\mathcal L_{\mathfrak V}$ is linear,
self-adjoint, and positive semidefinite on $V_{\Lambda}$ with respect to the
$H$ inner product. For a fixed tree $\Lambda$, however, the dependence of
$\mathcal L_{\mathfrak V}$ on the coefficient component
$\mathbf c_{\Lambda}$ may be nonlinear through the state-dependent
interaction weights.

\subsection{Structural assumptions on the adaptive operator family}

The symmetry, positive semidefiniteness, and annihilation of uniform
coefficient vectors by $L_{\mathfrak V}$ follow directly from
\eqref{eq:symmetric_weights}--\eqref{eq:zero_self_weight} and therefore need
not be imposed independently. To control the adaptive operator family
uniformly, we make the following additional assumptions.

\paragraph{Uniformly bounded interactions.}
There exists $\omega_{\ast}>0$, independent of the adaptive state, such that
\begin{align}
0
\leq
\omega_{\lambda\mu}
\left(
\mathfrak V
\right)
\leq
\omega_{\ast}
\label{eq:uniform_weight_bound}
\end{align}
for every $\mathfrak V\in\mathbb A$ and every
$\lambda,\mu\in\Lambda$.

\paragraph{Uniform interaction degree.}
There exists an integer $d_{\ast}<\infty$, independent of the adaptive state,
such that
\begin{align}
\#\left\{
\mu\in\Lambda
:
\omega_{\lambda\mu}
\left(
\mathfrak V
\right)
>
0
\right\}
\leq
d_{\ast},
\qquad
\lambda\in\Lambda.
\label{eq:uniform_interaction_degree}
\end{align}

Assumption~\eqref{eq:uniform_interaction_degree} restricts the admissible
interaction geometry to uniformly local graph constructions.  In particular,
an unrestricted all-to-all coupling of overlapping basis functions across
arbitrarily many levels is excluded unless it can be shown to satisfy the same
uniform degree bound.
Thus each active basis function interacts with at most $d_{\ast}$ other
active functions, uniformly over the admissible state space. In particular,
the associated weighted graph Laplacians have uniformly bounded row
sparsity.

\paragraph{Refinement consistency.}
Let $\Lambda\subseteq\Gamma$, and write
$\mathfrak V_{\Lambda}
=
\mathfrak V(\Lambda,\mathbf c_{\Lambda})$
and
$\mathfrak V_{\Gamma}
=
\mathfrak V(\Gamma,\mathbf c_{\Gamma})$.
We call $\mathfrak V_{\Gamma}$ a compatible nested refinement of
$\mathfrak V_{\Lambda}$ if
\begin{align}
u_{\mathfrak V_{\Gamma}}
=
\mathcal Q_{\Gamma\leftarrow\Lambda}
u_{\mathfrak V_{\Lambda}}.
\label{eq:compatible_refinement}
\end{align}
For such a pair, define the refinement-commutator defect by
\begin{align}
\varepsilon_{\mathrm{rc}}
\left(
\mathfrak V_{\Lambda},
\mathfrak V_{\Gamma}
\right)
=
\sup_{
\substack{
u\in V_{\Lambda}\\
u\neq0
}
}
\frac{
\left\|
\mathcal L_{\mathfrak V_{\Gamma}}
\mathcal Q_{\Gamma\leftarrow\Lambda}u
-
\mathcal Q_{\Gamma\leftarrow\Lambda}
\mathcal L_{\mathfrak V_{\Lambda}}u
\right\|_{H}
}{
\|u\|_{H}
}.
\label{eq:refinement_commutator_defect}
\end{align}
The operator family is called refinement consistent along a nested sequence
of compatible states if
\begin{align}
\varepsilon_{\mathrm{rc}}
\left(
\mathfrak V_{\Lambda_j},
\mathfrak V_{\Lambda_{j+1}}
\right)
\longrightarrow
0
\qquad
\text{as}
\qquad
j\longrightarrow\infty.
\label{eq:refinement_consistency_limit}
\end{align}

\subsection{Elementary operator properties}

\begin{proposition}
\label{prop:operator_properties}
Let $\mathfrak V\in\mathbb A$ be fixed and assume
\eqref{eq:symmetric_weights}--\eqref{eq:zero_self_weight}. Then
$L_{\mathfrak V}$ is symmetric and positive semidefinite. Moreover, for all
$\mathbf x,\mathbf y\in\mathbb R^{N_{\Lambda}}$,
\begin{align}
\mathbf x^{\mathrm T}
L_{\mathfrak V}
\mathbf y
=
a_{\mathfrak V}
\left(
\mathbf x,
\mathbf y
\right),
\label{eq:operator_form_identity}
\end{align}
and, in particular,
\begin{align}
\mathbf x^{\mathrm T}
L_{\mathfrak V}
\mathbf x
=
\frac{1}{2}
\sum_{\lambda,\mu\in\Lambda}
\omega_{\lambda\mu}
\left(
\mathfrak V
\right)
\left|
x_{\lambda}-x_{\mu}
\right|^{2}
\geq
0.
\label{eq:quadratic_identity}
\end{align}
The uniform coefficient vector belongs to the kernel:
\begin{align}
L_{\mathfrak V}\mathbf 1
=
\mathbf 0.
\label{eq:operator_null_constant}
\end{align}
The function-space operator $\mathcal L_{\mathfrak V}$ defined in
\eqref{eq:function_operator} is therefore self-adjoint and positive
semidefinite on $V_{\Lambda}$ with respect to the $H$ inner product.
\end{proposition}

\begin{proof}
Symmetry follows from \eqref{eq:symmetric_weights}. Expanding
$\mathbf x^{\mathrm T}L_{\mathfrak V}\mathbf y$ and exchanging the indices
yields \eqref{eq:operator_form_identity}. Taking
$\mathbf y=\mathbf x$ gives \eqref{eq:quadratic_identity}, and
nonnegativity follows from the nonnegative interaction weights. Equation
\eqref{eq:operator_null_constant} follows because every coefficient
difference vanishes for $\mathbf 1$. Finally,
\eqref{eq:function_operator_variational} transfers symmetry and
nonnegativity of the interaction form to the $H$ inner product on
$V_{\Lambda}$.
\end{proof}

Under \eqref{eq:uniform_weight_bound} and
\eqref{eq:uniform_interaction_degree}, the coefficient-space operators are
uniformly bounded. In particular,
\begin{align}
\left\|
L_{\mathfrak V}
\right\|_{\ell^{2}(\Lambda)\rightarrow\ell^{2}(\Lambda)}
\leq
2d_{\ast}\omega_{\ast},
\qquad
\mathfrak V\in\mathbb A.
\label{eq:uniform_operator_bound}
\end{align}

Together with the lower Riesz bound in \eqref{eq:Riesz_bounds}, this yields
the corresponding uniform estimate for the function-space operators,
\begin{align}
\left\|
\mathcal L_{\mathfrak V}
\right\|_{H\rightarrow H}
\leq
\frac{2d_{\ast}\omega_{\ast}}{A_{\Psi}},
\qquad
\mathfrak V\in\mathbb A.
\label{eq:uniform_function_operator_bound}
\end{align}
Here the norm is understood on the active space
$V_{\Lambda}\subset H$ associated with the state $\mathfrak V$.

Finally, define the positive-weight interaction graph by retaining precisely
those pairs for which
$\omega_{\lambda\mu}(\mathfrak V)>0$. If this graph is connected, then
\begin{align}
\ker
L_{\mathfrak V}
=
\operatorname{span}
\left\{
\mathbf 1
\right\}.
\label{eq:connected_kernel}
\end{align}

\section{Variational Formulation}
\label{sec:variational}

\subsection{Fidelity functional}

Let
\begin{align}
\mathcal F
:
H
\longrightarrow
(-\infty,+\infty]
\label{eq:general_fidelity_mapping}
\end{align}
be a proper, lower semicontinuous functional that is bounded from below.
The abstract variational formulation developed below requires only these
properties.

A principal example is the quadratic fidelity functional arising in inverse
problems. Let
\begin{align}
A
:
H
\longrightarrow
Y
\label{eq:forward_operator}
\end{align}
be a bounded linear operator into a Hilbert data space $Y$. For an exact
state $u^{\dagger}\in H$, suppose that the measured data satisfy
\begin{align}
f^{\delta}
=
Au^{\dagger}
+
\xi,
\qquad
\|\xi\|_{Y}
\leq
\delta,
\label{eq:noisy_data}
\end{align}
with noise $\xi\in Y$. The corresponding fidelity functional is
\begin{align}
\mathcal F
\left(
u
\right)
=
\frac{1}{2}
\left\|
Au-f^{\delta}
\right\|_{Y}^{2}.
\label{eq:fidelity_functional}
\end{align}
Since $A$ is bounded, \eqref{eq:fidelity_functional} is continuous,
nonnegative, and convex on $H$.

\subsection{Adaptive regularization energy}

Let
$\mathfrak V=\mathfrak V(\Lambda,\mathbf c_{\Lambda})\in\mathbb A$,
and let $u_{\mathfrak V}\in V_{\Lambda}$ denote the represented function
defined in \eqref{eq:state_function}. For a parameter $\alpha>0$, define the
intrinsic diffusion regularizer by
\begin{align}
\mathcal R_{\mathrm{diff}}
\left(
\mathfrak V
\right)
=
\alpha
\mathcal D_{\mathfrak V}
\left(
\mathbf c_{\Lambda}
\right).
\label{eq:diffusion_regularizer}
\end{align}
By \eqref{eq:intrinsic_Dirichlet_energy} and
\eqref{eq:operator_form_identity}, this term is nonnegative.

For the variational analysis, we assume that, for every fixed admissible tree
$\Lambda$ and every $\lambda,\mu\in\Lambda$, the mapping
\begin{align}
\mathbf c_{\Lambda}
\longmapsto
\omega_{\lambda\mu}
\left(
\mathfrak V(\Lambda,\mathbf c_{\Lambda})
\right)
\label{eq:weight_coefficient_continuity}
\end{align}
is continuous. Since $\Lambda$ is finite, this assumption implies that
$\mathcal R_{\mathrm{diff}}$ is continuous as a function of
$\mathbf c_{\Lambda}\in\mathbb R^{N_{\Lambda}}$ for every fixed admissible
tree. The coefficient-sensitive kernels in
\eqref{eq:coefficient_kernel} and
\eqref{eq:gaussian_coefficient_kernel} satisfy this assumption.

To penalize the size of the adaptive representation and, through suitable
level-dependent costs, its refinement depth, introduce the tree-complexity
functional
\begin{align}
\mathcal C
\left(
\Lambda
\right)
=
\sum_{\lambda\in\Lambda}
\rho_{\ell_{\lambda}},
\label{eq:tree_complexity}
\end{align}
where the level costs satisfy
\begin{align}
\rho_{\ell}
\geq
\rho_{\ast}
>
0,
\qquad
\ell\geq0,
\label{eq:positive_level_cost}
\end{align}
for a constant $\rho_{\ast}$ independent of the level.
A useful choice is
\begin{align}
\rho_{\ell}
=
1+\gamma\ell,
\qquad
\gamma\geq0,
\label{eq:linear_level_cost}
\end{align}
which penalizes the number of active indices and, when $\gamma>0$, introduces
an additional cost for increasingly fine resolution levels.

A coefficient sparsity term may additionally be included. For $\beta\geq0$,
let
\begin{align}
\mathcal R_{\mathrm{sp}}
\left(
\mathbf c_{\Lambda}
\right)
=
\beta
\sum_{\lambda\in\Lambda}
\nu_{\lambda}
\,
\varphi
\left(
|c_{\lambda}|
\right),
\label{eq:sparsity_regularizer}
\end{align}
where the coefficient weights satisfy
\begin{align}
\nu_{\lambda}
\geq
\nu_{\ast}
>
0
\label{eq:positive_sparsity_weights}
\end{align}
uniformly over the ambient index set, and
$\varphi:[0,\infty)\rightarrow[0,\infty)$ is convex, lower
semicontinuous, coercive, and satisfies $\varphi(0)=0$. Typical choices are
\begin{align}
\varphi(r)
&=
r,
\label{eq:l1_phi}
\\
\varphi_{\epsilon}(r)
&=
\sqrt{r^{2}+\epsilon^{2}}-\epsilon,
\qquad
\epsilon>0.
\label{eq:smooth_l1_phi}
\end{align}

The total adaptive energy is defined directly on the adaptive state space.
For $\alpha>0$, $\beta\geq0$, and $\tau>0$, let
\begin{align}
\boxed{
\begin{aligned}
\mathcal E
\left(
\mathfrak V
\right)
=
&
\mathcal F
\left(
u_{\mathfrak V}
\right)
\\
&
+
\alpha
\mathcal D_{\mathfrak V}
\left(
\mathbf c_{\Lambda}
\right)
\\
&
+
\beta
\sum_{\lambda\in\Lambda}
\nu_{\lambda}
\,
\varphi
\left(
|c_{\lambda}|
\right)
\\
&
+
\tau
\mathcal C
\left(
\Lambda
\right).
\end{aligned}
}
\label{eq:total_adaptive_energy}
\end{align}
The four contributions in \eqref{eq:total_adaptive_energy} represent,
respectively, the fidelity term, intrinsic adaptive diffusion, coefficient
sparsity, and adaptive-tree complexity.

Because $u_{\mathfrak V}$, $V_{\Lambda}$, $\Phi_{\Lambda}$, and
$\mathcal T_{\Lambda}$ are determined by the adaptive state, no additional
compatibility constraint between the represented function and the
coefficient vector is required.
The functional
$\mathcal E:\mathbb A\rightarrow(-\infty,+\infty]$ is defined on the
adaptive state space $\mathbb A$, which is not a linear space.
For a fixed tree $\Lambda$, its
dependence on the coefficient vector may be nonconvex when the interaction
weights depend nonlinearly on $\mathbf c_{\Lambda}$. Across different trees,
the discrete change of approximation space gives the variational problem an
additional combinatorial character.

When the interaction weights depend on $\mathbf c_{\Lambda}$, the quantity
$\mathcal D_{\mathfrak V}(\mathbf c_{\Lambda})$ must be distinguished from
the quadratic energy associated with a frozen adaptive state. In particular,
the full variation of
$\mathcal D_{\mathfrak V}(\mathbf c_{\Lambda})$ with respect to the
coefficient component generally contains contributions arising from the
coefficient dependence of the weights. Consequently,
$L_{\mathfrak V}\mathbf c_{\Lambda}$ represents the diffusion action for a
fixed state, but it is not, in general, the complete gradient of the
state-dependent regularization energy. The discrete variational evolution
introduced below is therefore formulated directly in terms of
$\mathcal E$ rather than by identifying its full gradient with
$L_{\mathfrak V}\mathbf c_{\Lambda}$.

\section{Coupled Variational Evolution}
\label{sec:discrete_evolution}

\subsection{Motivation for a time-discrete formulation}

For a fixed admissible tree, the coefficient variables belong to a
finite-dimensional vector space and may vary continuously. Refinement and
coarsening, by contrast, modify the discrete tree configuration. We therefore
do not introduce a formal time derivative of the tree variable. Instead, the
coupled evolution is defined through successive variational updates on the
adaptive state space.

Let
\begin{align}
0=t_{0}<t_{1}<\cdots<t_{K}=T,
\qquad
h_{n}=t_{n+1}-t_{n}>0,
\label{eq:time_partition}
\end{align}
and let
\begin{align}
\mathfrak V^{n}
=
\mathfrak V
\left(
\Lambda^{n},
\mathbf c^{n}
\right)
\in
\mathbb A
\label{eq:discrete_state}
\end{align}
denote the adaptive state at step $n$. We write
\begin{align}
u^{n}
=
u_{\mathfrak V^{n}}
\in
V_{\Lambda^{n}}
\label{eq:discrete_represented_function}
\end{align}
for its represented function.

\subsection{Admissible local tree modifications}

For each $\Lambda\in\mathbb T_{\mathrm{adm}}$, let
\begin{align}
\mathcal N
\left(
\Lambda
\right)
\subset
\mathbb T_{\mathrm{adm}}
\label{eq:tree_neighbourhood}
\end{align}
be a finite set of admissible trees reachable from $\Lambda$ by a prescribed
finite number of local refinement or coarsening operations. We require
\begin{align}
\Lambda
\in
\mathcal N
\left(
\Lambda
\right),
\label{eq:tree_neighbourhood_identity}
\end{align}
so that retaining the current tree is always an admissible choice.

For an integer $m\geq1$, a typical neighbourhood is
\begin{align}
\mathcal N_{m}
\left(
\Lambda
\right)
=
\left\{
\Gamma\in\mathbb T_{\mathrm{adm}}
:
\Gamma
\text{ is reachable from }
\Lambda
\text{ by at most }
m
\text{ admissible local tree moves}
\right\}.
\label{eq:m_tree_neighbourhood}
\end{align}
We assume that the underlying multiresolution hierarchy has finite branching,
so that $\mathcal N_{m}(\Lambda)$ is finite for every finite $\Lambda$ and
finite $m$.

\subsection{Joint variational time step}

For $\Gamma\in\mathcal N(\Lambda^{n})$ and $v\in V_{\Gamma}$, let
$\mathbf c_{\Gamma}(v)$ denote the coefficient vector of $v$ in
$\Phi_{\Gamma}$ and define the corresponding candidate adaptive state by
\begin{align}
\mathfrak V_{\Gamma}(v)
=
\mathfrak V
\left(
\Gamma,
\mathbf c_{\Gamma}(v)
\right).
\label{eq:candidate_adaptive_state}
\end{align}

Given $\mathfrak V^{n}$, select a minimizing tree--function pair by
\begin{align}
\boxed{
\left(
\Lambda^{n+1},
u^{n+1}
\right)
\in
\operatorname*{argmin}_{
\substack{
\Gamma\in\mathcal N(\Lambda^{n})\\
v\in V_{\Gamma}
}
}
\left\{
\frac{1}{2h_{n}}
\left\|
v-
\mathcal Q_{\Gamma\leftarrow\Lambda^{n}}
u^{n}
\right\|_{H}^{2}
+
\mathcal E
\left(
\mathfrak V_{\Gamma}(v)
\right)
\right\}.
}
\label{eq:joint_minimizing_movement}
\end{align}
The corresponding updated adaptive state is
\begin{align}
\mathfrak V^{n+1}
=
\mathfrak V_{\Lambda^{n+1}}
\left(
u^{n+1}
\right).
\label{eq:new_state_from_minimizer}
\end{align}

Equation~\eqref{eq:joint_minimizing_movement} defines the coupled
time-discrete variational update of the adaptive system. The first term penalizes the
change of the represented function after transfer to a candidate adaptive
space, while the second term evaluates the complete state-dependent energy
defined in \eqref{eq:total_adaptive_energy}. Thus the represented function
and the adaptive tree are selected simultaneously.

\subsection{Split implementation}

For practical computation, the joint minimization in
\eqref{eq:joint_minimizing_movement} may be approximated by a coefficient
update followed by a tree-adaptation step.

First, keeping the current tree fixed, compute
\begin{align}
\widetilde u^{n+1}
\in
\operatorname*{argmin}_{v\in V_{\Lambda^{n}}}
\left\{
\frac{1}{2h_{n}}
\left\|
v-u^{n}
\right\|_{H}^{2}
+
\mathcal E
\left(
\mathfrak V_{\Lambda^{n}}(v)
\right)
\right\}.
\label{eq:coefficient_step}
\end{align}

Next, for each candidate tree
$\Gamma\in\mathcal N(\Lambda^{n})$, define
\begin{align}
v_{\Gamma}^{n+1}
=
\mathcal Q_{\Gamma\leftarrow\Lambda^{n}}
\widetilde u^{n+1},
\label{eq:candidate_transferred_function}
\end{align}
and let
$\mathfrak V_{\Gamma}(v_{\Gamma}^{n+1})$ be the associated candidate state.

The new tree is then selected by
\begin{align}
\Lambda^{n+1}
\in
\operatorname*{argmin}_{\Gamma\in\mathcal N(\Lambda^{n})}
\left\{
\mathcal E
\left(
\mathfrak V_{\Gamma}
\left(
v_{\Gamma}^{n+1}
\right)
\right)
+
\frac{\zeta}{2}
d_{\mathbb T}^{2}
\left(
\Gamma,\Lambda^{n}
\right)
\right\},
\label{eq:tree_step}
\end{align}
where $\zeta\geq0$ controls the cost of modifying the adaptive tree. We use
the tree-change distance
\begin{align}
d_{\mathbb T}
\left(
\Gamma,\Lambda
\right)
=
\left|
\Gamma\triangle\Lambda
\right|,
\label{eq:tree_distance}
\end{align}
with $\triangle$ denoting symmetric set difference.

In the two-dimensional Haar implementation used for the numerical
experiments, the tree-change distance is evaluated through the symmetric
difference of refined spatial cells rather than individual wavelet indices.
The parameter $\zeta$ used in the numerical calculations is calibrated
according to this cell-based convention.

Finally, set
\begin{align}
u^{n+1}
=
v_{\Lambda^{n+1}}^{n+1}
=
\mathcal Q_{\Lambda^{n+1}\leftarrow\Lambda^{n}}
\widetilde u^{n+1},
\label{eq:transfer_after_adaptation}
\end{align}
and define
\begin{align}
\mathfrak V^{n+1}
=
\mathfrak V_{\Lambda^{n+1}}
\left(
u^{n+1}
\right).
\label{eq:split_updated_state}
\end{align}

\subsection{Fixed-tree optimality and frozen-weight linearization}

For a fixed admissible tree $\Lambda$, define the reduced coefficient-space
energy
\begin{align}
\mathcal E_{\Lambda}
\left(
\mathbf c
\right)
=
\mathcal E
\left(
\mathfrak V(\Lambda,\mathbf c)
\right).
\label{eq:fixed_tree_energy}
\end{align}
Using the Gram matrix defined in \eqref{eq:Gram_matrix}, the fixed-tree
coefficient step corresponding to \eqref{eq:coefficient_step} may be written
as
\begin{align}
\widetilde{\mathbf c}^{\,n+1}
\in
\operatorname*{argmin}_{\mathbf c\in\mathbb R^{N_{\Lambda^{n}}}}
\left\{
\frac{1}{2h_{n}}
\left(
\mathbf c-\mathbf c^{n}
\right)^{\mathrm T}
G_{\Lambda^{n}}
\left(
\mathbf c-\mathbf c^{n}
\right)
+
\mathcal E_{\Lambda^{n}}
\left(
\mathbf c
\right)
\right\}.
\label{eq:coefficient_step_coordinates}
\end{align}

If $\mathcal E_{\Lambda^{n}}$ is differentiable at the minimizer, the
corresponding first-order optimality condition is
\begin{align}
\frac{1}{h_{n}}
G_{\Lambda^{n}}
\left(
\widetilde{\mathbf c}^{\,n+1}
-
\mathbf c^{n}
\right)
+
\nabla
\mathcal E_{\Lambda^{n}}
\left(
\widetilde{\mathbf c}^{\,n+1}
\right)
=
\mathbf 0.
\label{eq:exact_fixed_tree_optimality}
\end{align}
If the coefficient sparsity term is convex but nonsmooth while the remaining
fixed-tree energy is differentiable, the corresponding first-order condition
is understood as a subdifferential inclusion, with the derivative of the
sparsity term replaced by its convex subdifferential.

When the interaction weights depend on the coefficient vector, the gradient
of the intrinsic diffusion energy contains the additional weight-variation
terms discussed in Section~\ref{sec:variational}. Consequently,
$L_{\mathfrak V}\mathbf c$ alone does not represent the full gradient of the
state-dependent diffusion energy.

A computationally simpler frozen-weight approximation is obtained by
evaluating the interaction weights at the previous state
$\mathfrak V^{n}$ during the coefficient update. We denote the resulting
approximate coefficient vector by
$\widehat{\mathbf c}^{\,n+1}$, thereby distinguishing it from the exact
fixed-tree minimizer $\widetilde{\mathbf c}^{\,n+1}$ in
\eqref{eq:coefficient_step_coordinates}. For the quadratic fidelity
functional in \eqref{eq:fidelity_functional}, define
\begin{align}
A_{\Lambda}
:
\mathbb R^{N_{\Lambda}}
\longrightarrow
Y
\label{eq:restricted_forward_operator}
\end{align}
as the coefficient representation of the restriction of $A$ to
$V_{\Lambda}$, using \eqref{eq:adaptive_expansion}. Let
$A_{\Lambda}^{\ast}$ denote its adjoint with respect to the Euclidean
coefficient inner product and the inner product of $Y$.

If, in addition, the coefficient regularization is taken to be quadratic in
this frozen-weight approximation, replacing the sparsity contribution by
\begin{align}
\mathcal S_{\Lambda}
\left(
\mathbf c
\right)
=
\frac{1}{2}
\mathbf c^{\mathrm T}
R_{\Lambda}
\mathbf c,
\qquad
R_{\Lambda}
\succeq
0,
\label{eq:quadratic_coefficient_regularization}
\end{align}
then the frozen-weight first-order condition becomes
\begin{align}
\frac{1}{h_{n}}
G_{\Lambda^{n}}
\left(
\widehat{\mathbf c}^{\,n+1}
-
\mathbf c^{n}
\right)
+
A_{\Lambda^{n}}^{\ast}
\left(
A_{\Lambda^{n}}
\widehat{\mathbf c}^{\,n+1}
-
f^{\delta}
\right)
+
\alpha
L_{\mathfrak V^{n}}
\widehat{\mathbf c}^{\,n+1}
+
\beta
R_{\Lambda^{n}}
\widehat{\mathbf c}^{\,n+1}
=
\mathbf 0.
\label{eq:frozen_weight_optimality}
\end{align}
Equivalently,
\begin{align}
\Big[
G_{\Lambda^{n}}
+
h_{n}
\big(
A_{\Lambda^{n}}^{\ast}A_{\Lambda^{n}}
+
\alpha L_{\mathfrak V^{n}}
+
\beta R_{\Lambda^{n}}
\big)
\Big]
\widehat{\mathbf c}^{\,n+1}
=
G_{\Lambda^{n}}
\mathbf c^{n}
+
h_{n}
A_{\Lambda^{n}}^{\ast}
f^{\delta}.
\label{eq:implicit_linear_update}
\end{align}
Since $G_{\Lambda^{n}}$ is symmetric positive definite, while
$A_{\Lambda^{n}}^{\ast}A_{\Lambda^{n}}$,
$L_{\mathfrak V^{n}}$, and $R_{\Lambda^{n}}$ are positive semidefinite, the
matrix on the left-hand side of \eqref{eq:implicit_linear_update} is symmetric
positive definite. Hence the frozen-weight coefficient update admits a unique
solution for every $h_{n}>0$.
If the active basis is $H$-orthonormal, then
$G_{\Lambda^{n}}=I$, and \eqref{eq:implicit_linear_update} reduces to the
usual Euclidean-coordinate implicit update.

\section{Existence, Stability, and Energy Decay}
\label{sec:analysis}

The results in this section concern the exact joint time-discrete variational
scheme in \eqref{eq:joint_minimizing_movement}, unless stated otherwise.
The split formulation in Section~\ref{sec:discrete_evolution} is treated
separately below. No energy-dissipation property is assumed for the
frozen-weight linearization without an additional acceptance criterion.

\subsection{Existence of a discrete update}

\begin{theorem}[Existence of the adaptive variational step]
\label{thm:discrete_existence}
Let $\mathfrak V^{n}\in\mathbb A$ satisfy
\begin{align}
\mathcal E
\left(
\mathfrak V^{n}
\right)
<
\infty.
\label{eq:finite_current_energy}
\end{align}
Assume that $\mathcal N(\Lambda^{n})$ is finite and that, for every
$\Gamma\in\mathcal N(\Lambda^{n})$, the fixed-tree functional
\begin{align}
v
\longmapsto
\mathcal E
\left(
\mathfrak V_{\Gamma}(v)
\right)
\label{eq:fixed_tree_energy_functional}
\end{align}
is proper on $V_{\Gamma}$. Under the assumptions of
Section~\ref{sec:variational}, the minimization problem
\eqref{eq:joint_minimizing_movement} admits at least one solution.
\end{theorem}

\begin{proof}
Fix $\Gamma\in\mathcal N(\Lambda^{n})$. Since $V_{\Gamma}$ is finite
dimensional, the coefficient map
$v\mapsto\mathbf c_{\Gamma}(v)$ is continuous. The fidelity contribution is
lower semicontinuous, the intrinsic diffusion contribution is continuous by
\eqref{eq:weight_coefficient_continuity}, and the coefficient regularization
is lower semicontinuous. Hence
$v\mapsto\mathcal E(\mathfrak V_{\Gamma}(v))$ is lower semicontinuous on
$V_{\Gamma}$.

Moreover, $\mathcal E(\mathfrak V_{\Gamma}(\cdot))$ is bounded from below
on $V_{\Gamma}$, while the incremental $H$-norm term in
\eqref{eq:joint_minimizing_movement} is coercive on $V_{\Gamma}$. 
Therefore the complete incremental functional is proper,
lower semicontinuous, and coercive on the finite-dimensional space
$V_{\Gamma}$, and consequently attains its minimum.

Since $\mathcal N(\Lambda^{n})$ contains only finitely many candidate trees,
the smallest of the corresponding fixed-tree minima is attained. Hence
\eqref{eq:joint_minimizing_movement} admits at least one minimizing
tree--function pair.
\end{proof}

\subsection{Discrete energy inequality}

\begin{theorem}[Discrete energy dissipation]
\label{thm:energy_decay}
Let $(\Lambda^{n+1},u^{n+1})$ solve
\eqref{eq:joint_minimizing_movement}. Then
\begin{align}
\mathcal E
\left(
\mathfrak V^{n+1}
\right)
+
\frac{1}{2h_{n}}
\left\|
u^{n+1}
-
\mathcal Q_{\Lambda^{n+1}\leftarrow\Lambda^{n}}
u^{n}
\right\|_{H}^{2}
\leq
\mathcal E
\left(
\mathfrak V^{n}
\right).
\label{eq:discrete_energy_inequality}
\end{align}
\end{theorem}

\begin{proof}
By \eqref{eq:tree_neighbourhood_identity}, the current tree
$\Lambda^{n}$ is an admissible candidate in
\eqref{eq:joint_minimizing_movement}. Choosing
$\Gamma=\Lambda^{n}$ and $v=u^{n}$, the reproduction property
\eqref{eq:projection_reproduction} implies that the incremental distance
term vanishes. The candidate state is then precisely $\mathfrak V^{n}$.
Minimality of $(\Lambda^{n+1},u^{n+1})$ therefore gives
\eqref{eq:discrete_energy_inequality}.
\end{proof}

Summing \eqref{eq:discrete_energy_inequality} from $n=0$ to $K-1$ gives
\begin{align}
\mathcal E
\left(
\mathfrak V^{K}
\right)
+
\sum_{n=0}^{K-1}
\frac{1}{2h_{n}}
\left\|
u^{n+1}
-
\mathcal Q_{\Lambda^{n+1}\leftarrow\Lambda^{n}}
u^{n}
\right\|_{H}^{2}
\leq
\mathcal E
\left(
\mathfrak V^{0}
\right).
\label{eq:cumulative_energy_bound}
\end{align}
In particular, the sequence
$\{\mathcal E(\mathfrak V^{n})\}_{n=0}^{K}$ is nonincreasing, and the
accumulated transfer-corrected increments are controlled by the initial
energy.

\subsection{Energy stability of the variational split scheme}

The split scheme in \eqref{eq:coefficient_step}--\eqref{eq:split_updated_state}
also inherits an energy inequality when both subproblems are solved exactly.

\begin{proposition}[Energy stability of the split update]
\label{prop:split_energy_stability}
Let $\widetilde u^{n+1}$ solve \eqref{eq:coefficient_step}, and let
$\Lambda^{n+1}$ solve \eqref{eq:tree_step}. Then
\begin{align}
\mathcal E
\left(
\mathfrak V^{n+1}
\right)
+
\frac{1}{2h_{n}}
\left\|
\widetilde u^{n+1}-u^{n}
\right\|_{H}^{2}
+
\frac{\zeta}{2}
d_{\mathbb T}^{2}
\left(
\Lambda^{n+1},
\Lambda^{n}
\right)
\leq
\mathcal E
\left(
\mathfrak V^{n}
\right).
\label{eq:split_energy_inequality}
\end{align}
\end{proposition}

\begin{proof}
Using $v=u^{n}$ as a competitor in \eqref{eq:coefficient_step} gives the
fixed-tree energy inequality. In the subsequent tree minimization,
$\Gamma=\Lambda^{n}$ is admissible by
\eqref{eq:tree_neighbourhood_identity}. For this choice,
\eqref{eq:nested_injection} implies that the transferred function is
$\widetilde u^{n+1}$, while the tree-change penalty vanishes. Combining the
two minimizing properties yields \eqref{eq:split_energy_inequality}.
\end{proof}

\begin{corollary}[Finite tree activity]
\label{cor:finite_tree_activity}
Assume $\zeta>0$ and let
$\{\mathfrak V^{n}\}_{n\geq0}$ be generated by the exact variational split
scheme.  Let
\begin{align}
\mathcal E_{\inf}
=
\inf_{\mathfrak V\in\mathbb A}
\mathcal E(\mathfrak V)
>
-\infty.
\label{eq:energy_lower_bound_tree_activity}
\end{align}
Then the number of iterations at which the adaptive tree changes is finite.
More precisely,
\begin{align}
\#\left\{
n\geq0:
\Lambda^{n+1}\neq\Lambda^{n}
\right\}
\leq
\frac{
2\big(
\mathcal E(\mathfrak V^{0})-\mathcal E_{\inf}
\big)
}{
\zeta
}.
\label{eq:number_tree_changes_bound}
\end{align}
Consequently, there exists an index $n_{\ast}$ such that
\begin{align}
\Lambda^{n}
=
\Lambda^{n_{\ast}},
\qquad
n\geq n_{\ast}.
\label{eq:eventual_tree_stabilization}
\end{align}
\end{corollary}

\begin{proof}
Summing \eqref{eq:split_energy_inequality} and discarding the nonnegative
coefficient-increment terms gives
\begin{align}
\frac{\zeta}{2}
\sum_{n\geq0}
d_{\mathbb T}^{2}
\left(
\Lambda^{n+1},
\Lambda^{n}
\right)
\leq
\mathcal E(\mathfrak V^{0})
-
\mathcal E_{\inf}.
\end{align}
Whenever
$\Lambda^{n+1}\neq\Lambda^{n}$,
the symmetric-difference distance
\eqref{eq:tree_distance} is a positive integer and therefore satisfies
$d_{\mathbb T}^{2}\geq1$.
The number of nontrivial tree changes is consequently bounded by
\eqref{eq:number_tree_changes_bound}, which proves eventual stabilization.
\end{proof}

The estimate \eqref{eq:split_energy_inequality} applies to the exact
variational split update. The frozen-weight approximation introduced in
Section~\ref{sec:discrete_evolution} replaces the first minimization by a
linearized problem and therefore does not, by itself, imply decay of the
original state-dependent energy $\mathcal E$.

\subsection{Frozen-state intrinsic diffusion}

To study the intrinsic diffusion operator independently of tree adaptation,
fix a reference adaptive state
\begin{align}
\overline{\mathfrak V}
=
\mathfrak V
\left(
\Lambda,
\overline{\mathbf c}
\right)
\in
\mathbb A
\label{eq:frozen_reference_state}
\end{align}
and hold the associated interaction weights fixed. Consider the linear
evolution
\begin{align}
\frac{d u}{dt}
+
\mathcal L_{\overline{\mathfrak V}}u
=
0,
\qquad
u(t)\in V_{\Lambda}.
\label{eq:frozen_function_diffusion}
\end{align}
If $\mathbf c(t)$ denotes the coefficient vector of $u(t)$ in
$\Phi_{\Lambda}$, then \eqref{eq:function_operator} yields
\begin{align}
G_{\Lambda}
\frac{d\mathbf c}{dt}
+
L_{\overline{\mathfrak V}}
\mathbf c
=
\mathbf 0.
\label{eq:frozen_coefficient_diffusion}
\end{align}

This is the coefficient representation of the $H$-space diffusion generated
by $\mathcal L_{\overline{\mathfrak V}}$. In particular, the generator in
coefficient coordinates is
$G_{\Lambda}^{-1}L_{\overline{\mathfrak V}}$, rather than
$L_{\overline{\mathfrak V}}$ unless the active basis is
$H$-orthonormal.

\subsection{Conservation of the uniform-mode moment}

\begin{proposition}[Conservation of the uniform-mode moment]
\label{prop:uniform_mode_conservation}
Let $\mathbf c(t)$ solve \eqref{eq:frozen_coefficient_diffusion}. Then
\begin{align}
\frac{d}{dt}
\left(
\mathbf 1^{\mathrm T}
G_{\Lambda}
\mathbf c(t)
\right)
=
0.
\label{eq:uniform_mode_conservation}
\end{align}
\end{proposition}

\begin{proof}
Left-multiplying \eqref{eq:frozen_coefficient_diffusion} by
$\mathbf 1^{\mathrm T}$ and using
\eqref{eq:operator_null_constant} together with the symmetry of
$L_{\overline{\mathfrak V}}$ gives
\eqref{eq:uniform_mode_conservation}.
\end{proof}

The conserved quantity in \eqref{eq:uniform_mode_conservation} is the
$H$-inner-product moment associated with the function obtained by
synthesizing the uniform coefficient vector. It should not, in general, be
identified with physical mass or with the integral of $u$ over $\Omega$.
Such an interpretation requires an additional constant-reproduction
property of the active basis, which is not assumed here.

Accordingly, the uniform coefficient mode is an algebraic invariant of the
coefficient interaction graph, not a claim that the corresponding synthesized
function is spatially constant. If preservation of the physical constant
function is required in a particular application, the null space of the
interaction operator must instead be constructed to contain the coefficient
representation of that function.

\subsection{Contractivity on a frozen adaptive state}

\begin{proposition}[Contractivity]
\label{prop:contractivity}
For every $t\geq0$, the semigroup generated by
$-\mathcal L_{\overline{\mathfrak V}}$ satisfies
\begin{align}
\left\|
\exp
\left(
-t\mathcal L_{\overline{\mathfrak V}}
\right)
u
\right\|_{H}
\leq
\|u\|_{H},
\qquad
u\in V_{\Lambda}.
\label{eq:contractivity}
\end{align}
\end{proposition}

\begin{proof}
By Proposition~\ref{prop:operator_properties},
$\mathcal L_{\overline{\mathfrak V}}$ is self-adjoint and positive
semidefinite on $V_{\Lambda}$ with respect to the $H$ inner product.
Therefore its spectrum is contained in $[0,\infty)$, and the spectral multipliers of
$\exp(-t\mathcal L_{\overline{\mathfrak V}})$ belong to $[0,1]$ for
$t\geq0$. This proves \eqref{eq:contractivity}.
\end{proof}

\subsection{Long-time limit on a connected adaptive graph}

Assume that the positive-weight interaction graph associated with
$\overline{\mathfrak V}$ is connected. Consider the generalized eigenvalue
problem
\begin{align}
L_{\overline{\mathfrak V}}
\mathbf q_{j}
=
\lambda_{j}
G_{\Lambda}
\mathbf q_{j},
\qquad
j=1,\ldots,N_{\Lambda}.
\label{eq:generalized_eigenproblem}
\end{align}
The eigenvectors may be chosen $G_{\Lambda}$-orthonormal, and the generalized
eigenvalues satisfy
\begin{align}
0
=
\lambda_{1}
<
\lambda_{2}
\leq
\cdots
\leq
\lambda_{N_{\Lambda}}.
\label{eq:ordered_spectrum}
\end{align}
The zero eigenvalue is simple by \eqref{eq:connected_kernel}, with associated
eigenspace $\operatorname{span}\{\mathbf 1\}$.

Define
\begin{align}
\overline c_{G}
=
\frac{
\mathbf 1^{\mathrm T}
G_{\Lambda}
\mathbf c(0)
}{
\mathbf 1^{\mathrm T}
G_{\Lambda}
\mathbf 1
}.
\label{eq:generalized_coefficient_average}
\end{align}
Then the solution of \eqref{eq:frozen_coefficient_diffusion} satisfies
\begin{align}
\left\|
\mathbf c(t)
-
\overline c_{G}\mathbf 1
\right\|_{G_{\Lambda}}
\leq
\exp
\left(
-\lambda_{2}t
\right)
\left\|
\mathbf c(0)
-
\overline c_{G}\mathbf 1
\right\|_{G_{\Lambda}},
\label{eq:long_time_decay}
\end{align}
where
\begin{align}
\|\mathbf x\|_{G_{\Lambda}}^{2}
=
\mathbf x^{\mathrm T}
G_{\Lambda}
\mathbf x.
\label{eq:Gram_norm}
\end{align}
Hence the frozen-state diffusion converges exponentially to the
$H$-orthogonal projection of the initial condition onto
$\ker\mathcal L_{\overline{\mathfrak V}}$, with decay rate determined by the
generalized spectral gap $\lambda_{2}$.

The limiting state in \eqref{eq:long_time_decay} should therefore be
interpreted as the projection onto the algebraic null mode of the frozen
coefficient graph; no physical mass-equilibration interpretation is implied.

If the active basis is $H$-orthonormal, then $G_{\Lambda}=I$ and
\eqref{eq:generalized_coefficient_average} reduces to the ordinary arithmetic
mean of the initial coefficient vector.

\section{Adaptive Multiresolution Diffusion for Inverse Problems}
\label{sec:inverse_problems}

\subsection{General inverse problem}

Let $u^{\dagger}\in H$ be an exact state satisfying
\begin{align}
A u^{\dagger}
=
f,
\label{eq:exact_inverse_problem}
\end{align}
and let the noisy data satisfy \eqref{eq:noisy_data}. Throughout this section,
the fidelity term is the quadratic functional defined in
\eqref{eq:fidelity_functional}.

The adaptive reconstruction is formulated directly on the state space
$\mathbb A$. For parameters $\alpha>0$, $\beta\geq0$, and $\tau>0$, seek
\begin{align}
\boxed{
\mathfrak V_{\alpha,\beta,\tau}^{\delta}
\in
\operatorname*{argmin}_{\mathfrak V\in\mathbb A}
\mathcal E
\left(
\mathfrak V
\right).
}
\label{eq:adaptive_inverse_problem}
\end{align}
Writing
\begin{align}
\mathfrak V_{\alpha,\beta,\tau}^{\delta}
=
\mathfrak V
\left(
\Lambda_{\alpha,\beta,\tau}^{\delta},
\mathbf c_{\alpha,\beta,\tau}^{\delta}
\right),
\label{eq:adaptive_inverse_state}
\end{align}
the corresponding reconstructed function is
\begin{align}
u_{\alpha,\beta,\tau}^{\delta}
=
u_{\mathfrak V_{\alpha,\beta,\tau}^{\delta}}
\in
V_{\Lambda_{\alpha,\beta,\tau}^{\delta}}.
\label{eq:adaptive_inverse_reconstruction}
\end{align}
Thus the minimization determines simultaneously the reconstructed function,
the active multiresolution tree, the state-dependent intrinsic diffusion
operator, and the number
$N_{\Lambda_{\alpha,\beta,\tau}^{\delta}}$ of active coefficients.

The formulation in \eqref{eq:adaptive_inverse_problem} is global in the
adaptive state space, in contrast to the local tree neighbourhood used in
the time-discrete evolution of Section~\ref{sec:discrete_evolution}.
Accordingly, existence of a global minimizer requires control of both the
tree complexity and the coefficient variables.

\begin{theorem}[Existence of an adaptive inverse reconstruction]
\label{thm:inverse_reconstruction_existence}
Assume that the multiresolution hierarchy has finite branching and that the
conditions of Section~\ref{sec:variational} hold. If
\begin{align}
\beta>0,
\qquad
\tau>0,
\label{eq:inverse_positive_regularization}
\end{align}
then the minimization problem \eqref{eq:adaptive_inverse_problem} admits at
least one solution.
\end{theorem}

\begin{proof}
Let $\{\mathfrak V_{k}\}_{k\geq1}\subset\mathbb A$ be a minimizing sequence,
with
$\mathfrak V_{k}=\mathfrak V(\Lambda_{k},\mathbf c_{k})$.
For the quadratic fidelity considered here, all contributions to
\eqref{eq:total_adaptive_energy} are nonnegative. Hence boundedness of the energy values along the minimizing sequence and
\eqref{eq:positive_level_cost} imply
\begin{align}
\tau\rho_{\ast}
\left|
\Lambda_{k}
\right|
\leq
\tau
\mathcal C
\left(
\Lambda_{k}
\right)
\leq
\mathcal E
\left(
\mathfrak V_{k}
\right).
\label{eq:inverse_tree_cardinality_bound}
\end{align}
Therefore the cardinalities $|\Lambda_{k}|$ are uniformly bounded.
Tree admissibility implies that the presence of an index at a given
resolution level also requires its ancestral chain to be active. Together
with finite branching and the finite root set, the cardinality bound leaves
only finitely many admissible trees that can occur along the minimizing
sequence. After passage to a subsequence, we may therefore assume that
\begin{align}
\Lambda_{k}
=
\Lambda
\label{eq:inverse_fixed_tree_subsequence}
\end{align}
for a fixed admissible tree $\Lambda$.

On this fixed tree, \eqref{eq:positive_sparsity_weights}, the coercivity of
$\varphi$, and $\beta>0$ imply that the coefficient regularization is
coercive in $\mathbf c_{k}\in\mathbb R^{N_{\Lambda}}$. Hence
$\{\mathbf c_{k}\}$ is bounded. Passing to a further subsequence gives
convergence to some
$\mathbf c_{\ast}\in\mathbb R^{N_{\Lambda}}$.

The quadratic fidelity is continuous, the intrinsic diffusion contribution
is continuous by \eqref{eq:weight_coefficient_continuity}, and the
coefficient regularization is lower semicontinuous. Consequently,
\begin{align}
\mathcal E
\left(
\mathfrak V(\Lambda,\mathbf c_{\ast})
\right)
\leq
\liminf_{k\rightarrow\infty}
\mathcal E
\left(
\mathfrak V(\Lambda,\mathbf c_{k})
\right).
\label{eq:inverse_lower_semicontinuity}
\end{align}
Thus $\mathfrak V(\Lambda,\mathbf c_{\ast})$ is a minimizer of
\eqref{eq:adaptive_inverse_problem}.
\end{proof}

The condition $\beta>0$ in
Theorem~\ref{thm:inverse_reconstruction_existence} is a sufficient
coercivity condition rather than a necessary one. The case $\beta=0$ may
also admit a minimizer if the fidelity and intrinsic diffusion terms provide
sufficient coercivity on the relevant adaptive spaces.

\subsection{Adaptive multiresolution diffusion imaging}

For imaging applications, let $\Omega\subset\mathbb R^{2}$ and retain the
ambient space $H=L^{2}(\Omega)$ introduced in Section~\ref{sec:spaces}.
Different imaging modalities are represented by different choices of the
bounded forward operator $A$.

For denoising, one takes
\begin{align}
A
=
I.
\label{eq:denoising_model}
\end{align}
For deblurring, $A$ may be a bounded convolution operator
\begin{align}
A
=
K_{k},
\label{eq:deblurring_model}
\end{align}
associated with a prescribed point-spread function $k$. For computed
tomography, one may take
\begin{align}
A
=
\mathcal R,
\label{eq:CT_operator}
\end{align}
where $\mathcal R$ denotes the Radon transform between the chosen object and
data Hilbert spaces. For magnetic resonance imaging, a standard choice is
\begin{align}
A
=
P_{\Xi}\mathcal F,
\label{eq:MRI_operator}
\end{align}
where $\mathcal F$ denotes the Fourier transform and $P_{\Xi}$ the sampling
operator associated with the acquired frequency set $\Xi$. Within the
real-valued framework adopted here, complex MRI measurements may equivalently
be represented by their real and imaginary components in an appropriate real
data Hilbert space.

For any of these forward models, the adaptive multiresolution diffusion
imaging reconstruction is precisely the state-space minimization
\eqref{eq:adaptive_inverse_problem} with the total energy
\eqref{eq:total_adaptive_energy}. In particular, the intrinsic diffusion term
is evaluated at the candidate state itself. Hence, when the interaction
weights depend on the coefficients, the reconstruction retains the full
state dependence of the diffusion regularizer rather than replacing it by
the frozen-weight approximation introduced in
Section~\ref{sec:discrete_evolution}.

The formulation therefore couples data consistency, intrinsic diffusion,
coefficient regularization, and adaptive model complexity within a single
variational problem. The present analysis concerns reconstruction at fixed
noise level and fixed regularization parameters. Parameter-choice rules and
regularization convergence as $\delta\rightarrow0$ require additional
assumptions and are not invoked here.

The numerical study in Section~\ref{sec:numerical_validation} focuses on
denoising, $A=I$, in order to isolate the adaptive operator and
tree-evolution mechanisms. The deblurring, tomography, and Fourier-sampling
models above are included as admissible forward operators in the abstract
framework but are not numerically investigated in the present study.

\section{Adaptive Numerical Algorithm}
\label{sec:algorithm}

The variational formulation developed above leads naturally to a family of
adaptive numerical procedures based on the split scheme
\eqref{eq:coefficient_step}--\eqref{eq:split_updated_state}.
The theory requires only that a finite admissible candidate neighbourhood be
constructed at each iteration.  Algorithm~\ref{alg:AMDI} summarizes this
generic variational structure.  The denoising implementation used for the
numerical experiments employs a concrete data-driven bulk refinement and
coarsening rule described below; the local indicator
\eqref{eq:local_indicator} is presented as a more general alternative and is
not used to generate the numerical trees reported in Section~\ref{sec:numerical_validation}.

\begin{algorithm}[t]
\caption{Adaptive Multiresolution Diffusion Imaging (AMDI)}
\label{alg:AMDI}
\begin{algorithmic}[1]

\State Choose an initial admissible tree
$\Lambda^{0}\in\mathbb T_{\mathrm{adm}}$.

\State Choose an initial approximation
$u^{0}\in V_{\Lambda^{0}}$, determine its coefficient vector
$\mathbf c^{0}$, and set
$\mathfrak V^{0}=\mathfrak V(\Lambda^{0},\mathbf c^{0})$,
with $\mathcal E(\mathfrak V^{0})<\infty$.

\State Set $n\gets0$.

\While{the stopping criterion is not satisfied}

    \State Compute the fixed-tree variational update
    $\widetilde u^{n+1}$ from \eqref{eq:coefficient_step}.

    \State Let $\widetilde{\mathbf c}^{\,n+1}$ be the coefficient vector of
    $\widetilde u^{n+1}$ in $\Phi_{\Lambda^{n}}$ and define the intermediate
    fixed-tree state
    $\widetilde{\mathfrak V}^{\,n+1}
    =
    \mathfrak V(\Lambda^{n},\widetilde{\mathbf c}^{\,n+1})$.

\State Construct refinement and coarsening scores using a prescribed
application-dependent marking rule.

\State Generate a finite candidate neighbourhood
$\mathcal N(\Lambda^{n})$, including refinement, coarsening, and the
unchanged tree $\Lambda^{n}$.

\State Enforce tree admissibility for every candidate tree.

    \State Select $\Lambda^{n+1}$ by the variational tree update
    \eqref{eq:tree_step}.

    \State Transfer $\widetilde u^{n+1}$ to the selected adaptive space
    according to \eqref{eq:transfer_after_adaptation}.

    \State Construct the updated state
    $\mathfrak V^{n+1}$ according to \eqref{eq:split_updated_state}.

    \State $n\gets n+1$.

\EndWhile

\State \Return
$u^{n}$,
$\Lambda^{n}$,
and
$\mathcal L_{\mathfrak V^{n}}$.

\end{algorithmic}
\end{algorithm}

Because both variational subproblems in
Algorithm~\ref{alg:AMDI} are solved according to the split formulation of
Section~\ref{sec:discrete_evolution}, the energy estimate
\eqref{eq:split_energy_inequality} applies directly to the exact algorithm.

For the denoising experiments in Section~\ref{sec:numerical_validation},
where $A=I$, candidate trees are generated directly from the full-depth Haar
expansion of the noisy data.  For each candidate spatial cell $K$ at level
$\ell$, the refinement score is
\begin{align}
s_{\ell,K}
=
\sum_{q=1}^{2^{d}-1}
\left|
d_{\ell,K,q}
\left(
f^{\delta}
\right)
\right|^{2},
\label{eq:data_detail_score}
\end{align}
where $d_{\ell,K,q}(f^{\delta})$ denotes the corresponding Haar detail
coefficient of the data.  The implementation constructs at most three
candidate trees: the unchanged tree, a tree obtained by refining the highest
scoring fraction of admissible leaves, and a tree obtained by coarsening the
lowest scoring fraction of admissible prunable cells.  The reported
calculations use refinement and coarsening fractions of $0.10$.
The variational criterion \eqref{eq:tree_step} then selects among these
candidates.

This data-driven rule is specific to the present denoising prototype and is
not asserted to provide a directly transferable marking strategy for the
general forward operators considered in
Section~\ref{sec:inverse_problems}.

\subsection{Local refinement and coarsening indicators}

For more general forward problems, one possible local marking indicator may
combine coefficient magnitude, data mismatch, and local intrinsic-diffusion
variation. For
$\mathfrak V=\mathfrak V(\Lambda,\mathbf c_{\Lambda})$, define
\begin{align}
\eta_{\lambda}^{2}
=
&
\theta_{\mathrm c}
|c_{\lambda}|^{2}
+
\theta_{\mathrm d}
\left|
\left[
A_{\Lambda}^{\ast}
\left(
A_{\Lambda}\mathbf c_{\Lambda}
-
f^{\delta}
\right)
\right]_{\lambda}
\right|^{2}
\nonumber
\\
&
+
\theta_{\mathrm j}
\alpha
\sum_{\mu\in\Lambda}
\omega_{\lambda\mu}
\left(
\mathfrak V
\right)
|c_{\lambda}-c_{\mu}|^{2},
\qquad
\lambda\in\Lambda,
\label{eq:local_indicator}
\end{align}
where
$\theta_{\mathrm c},\theta_{\mathrm d},\theta_{\mathrm j}\geq0$
are scaling parameters.

The indicator \eqref{eq:local_indicator} is not used in the denoising
prototype reported in Section~\ref{sec:numerical_validation}; it is included
as a general candidate construction for applications in which a
state-dependent marking rule is desired.

The three contributions measure, respectively, the activity of the
coefficient, its contribution to the projected data-fidelity gradient, and
the local variation associated with the intrinsic diffusion graph. The
indicator \eqref{eq:local_indicator} is intended as a computational marking
quantity for indicator-based realizations; it is not asserted here to be a
reliable or efficient a posteriori error estimator for the full
state-dependent variational problem.

For an indicator-based realization, one may introduce thresholds satisfying
\begin{align}
0
\leq
\varepsilon_{\mathrm{crs}}
<
\varepsilon_{\mathrm{ref}},
\label{eq:indicator_threshold_ordering}
\end{align}
refinement may be triggered by
\begin{align}
\eta_{\lambda}
>
\varepsilon_{\mathrm{ref}}
\quad
\Longrightarrow
\quad
\text{mark $\lambda$ for refinement},
\label{eq:refine_rule}
\end{align}
whereas coarsening may be triggered by
\begin{align}
\eta_{\lambda}
<
\varepsilon_{\mathrm{crs}}
\quad
\Longrightarrow
\quad
\text{mark $\lambda$ for coarsening}.
\label{eq:coarsen_rule}
\end{align}
Such a separation introduces hysteresis and can reduce rapid oscillation
between refinement and coarsening. Refinement is performed
together with the tree closure required by the admissibility condition of
Section~\ref{sec:spaces}, whereas coarsening is permitted only when removal
of the corresponding indices preserves tree admissibility. In particular,
non-root indices may be removed only together with any descendants whose
retention would violate the ancestral closure condition.

In such an indicator-based realization, the marked indices may be used to
construct a finite candidate neighbourhood
$\mathcal N(\Lambda^{n})$. The unchanged tree is always retained as a
candidate in accordance with \eqref{eq:tree_neighbourhood_identity}.
Consequently, marking proposes local structural modifications, whereas the
variational criterion \eqref{eq:tree_step} determines which admissible tree
is actually accepted.

\subsection{Frozen-weight implementation}

When solving the full fixed-tree problem
\eqref{eq:coefficient_step} is computationally expensive, the frozen-weight
linearization of Section~\ref{sec:discrete_evolution} may be used instead.
At iteration $n$, the interaction weights are evaluated at
$\mathfrak V^{n}$, the matrix $L_{\mathfrak V^{n}}$ is assembled, and the
approximate coefficient vector
$\widehat{\mathbf c}^{\,n+1}$ is obtained from
\eqref{eq:implicit_linear_update}.

Because this linearized update does not automatically satisfy the energy
estimate proved for the exact variational coefficient step, the reported
implementation enforces energy monotonicity by backtracking along the
frozen-weight update direction.

Let $\widehat{\mathbf c}^{\,n+1}$ denote the full frozen-weight solution of
\eqref{eq:implicit_linear_update}.  Starting from $\vartheta_{0}=1$, define
\begin{align}
\mathbf c_{j}^{\,n+1}
=
\mathbf c^{n}
+
\vartheta_{j}
\left(
\widehat{\mathbf c}^{\,n+1}
-
\mathbf c^{n}
\right),
\qquad
\vartheta_{j}=2^{-j}.
\label{eq:frozen_weight_backtracking}
\end{align}
The first trial satisfying
\begin{align}
\mathcal E
\left(
\mathfrak V
\left(
\Lambda^{n},
\mathbf c_{j}^{\,n+1}
\right)
\right)
\leq
\mathcal E
\left(
\mathfrak V^{n}
\right)
+
10^{-12}
\label{eq:frozen_weight_energy_safeguard}
\end{align}
is accepted.  The numerical implementation tests at most twenty
backtracking reductions.  The accepted coefficient vector and its
represented function are then used in the subsequent tree-adaptation step.
Thus energy monotonicity of the numerical linearization is enforced
computationally and should be distinguished from the unconditional estimate
\eqref{eq:split_energy_inequality} proved for the exact variational split
step.

\subsection{Stopping criteria}

The iteration may be terminated when one or more prescribed conditions are
satisfied. Typical choices include a sufficiently small relative change in
the reconstructed function, a sufficiently small decrease of the adaptive
energy, stabilization of the active tree, or a prescribed maximum number of
iterations. When an estimate of the noise level is available, a
discrepancy-based stopping condition may additionally be imposed through the
data residual
\begin{align}
\left\|
A u^{n}
-
f^{\delta}
\right\|_{Y}.
\label{eq:algorithm_data_residual}
\end{align}
The particular stopping rule is application dependent and does not alter
the variational structure of the adaptive update.

\section{Numerical Validation}
\label{sec:numerical_validation}

The numerical experiments are designed to test the properties that motivate
the adaptive multiresolution diffusion construction rather than to optimize a
single reconstruction metric.  In particular, we examine: (i) the structural
properties of the assembled intrinsic operator, (ii) compatibility under
refinement, (iii) the discrete energy inequality, (iv) multiresolution
localization in a separate Haar approximation diagnostic, (v) the
reconstruction--representation tradeoff, (vi) convergence robustness and
initialization sensitivity, and (vii) the separate numerical roles of the
adaptive tree and the state-dependent interactions.  A higher-order
VAMPyR/MRCPP calculation is finally used as a cross-representation check of
the regional localization produced by an adaptive multiwavelet
representation; it is not used as an independent validation of the AMDI
diffusion operator.

Unless otherwise stated, image intensities are normalized to $[0,1]$ and the
spatial domain is $\Omega=[0,1]^2$.  The AMDI coefficient experiments use an
orthonormal adaptive Haar representation, so that the Gram matrix introduced
in Section~\ref{sec:intrinsic_operator} reduces to the identity,
$G_{\Lambda}=I$.  This choice makes the coefficient and $L^2$ norms identical
on the represented space and allows the operator diagnostics to be evaluated
without an additional basis-conditioning effect.  It is not required by the
theory developed in Sections~\ref{sec:spaces}--\ref{sec:analysis}.

For the numerical implementation, the coefficient-sensitive kernel is
$\kappa_{\mathrm c}(r)=\exp[-(r/\sigma_{\mathrm c})^{2}]$ and the
coefficient-regularization weights are
$\nu_{\lambda}=1+0.05\,\ell_{\lambda}$.
The two-dimensional Haar interaction graph contains same-support
cross-orientation couplings, matching-orientation nearest-neighbour
couplings, matching-orientation parent--child couplings, and root
scaling--wavelet couplings.
The parent--child refinement-decay factor in
\eqref{eq:refinement_decay_kernel} is included in all reported AMDI
calculations.  The denoising parameter sweep uses
$\gamma_{\mathrm r}\in\{0.20,0.35,0.60\}$; the dedicated
refinement-consistency and energy-dissipation diagnostics use the values
specified separately below.

For the denoising experiments, Gaussian noise with standard deviation
$\sigma=0.08$ is added to a $64\times64$ synthetic image, and the resulting
intensities are clipped to $[0,1]$.  
The competing methods are linear heat diffusion, Perona--Malik anisotropic
diffusion \cite{PeronaMalik1990}, wavelet soft thresholding
\cite{DonohoJohnstone1994,Donoho1995}, and total-variation (TV) denoising
using the Rudin--Osher--Fatemi model \cite{RudinOsherFatemi1992}.
Each baseline is tuned independently for RMSE
and SSIM on the same calibration realization; thus AMDI is not compared with
arbitrarily chosen baseline parameters. The AMDI calibration uses the full
quality--representation sweep on noise seed $11$, after which the selected
parameters are frozen before the subsequent repeatability and held-out
noise-realization assessments.

The reconstruction metrics are
\begin{align}
\operatorname{RMSE}
&=
\left(
\frac{1}{N}
\sum_{i=1}^{N}
|u_i-u_i^\dagger|^2
\right)^{1/2},
\label{eq:RMSE}
\\
\operatorname{PSNR}
&=
10\log_{10}
\left(
\frac{u_{\max}^2}{\operatorname{MSE}}
\right),
\label{eq:PSNR}
\\
\mathcal C_{\mathrm{rel}}
&=
\frac{N_{\Lambda}}{N_{\mathrm{full}}},
\label{eq:relative_complexity}
\end{align}
together with the structural similarity index
$\operatorname{SSIM}(u,u^\dagger)$.  The relative complexity
$\mathcal C_{\mathrm{rel}}$ is the fraction of active degrees of freedom
relative to the corresponding uniformly refined representation.
Thus, $\mathcal C_{\mathrm{rel}}$ measures the size of the active
representation; it is not intended as a measure of wall-clock time or total
arithmetic cost.

\subsection{Intrinsic-operator structure and refinement consistency}
\label{subsec:numerical_operator}

We first verify the structural identities established in
Section~\ref{sec:intrinsic_operator}.  For a representative one-dimensional
multiresolution state, the assembled intrinsic Laplacian satisfies
\begin{align}
\|L_{\mathfrak V}-L_{\mathfrak V}^{\mathrm T}\|_2
&=0,
\\
\|L_{\mathfrak V}\mathbf 1\|_2
&=6.94\times10^{-16}.
\end{align}
The smallest computed eigenvalue is
$-1.16\times10^{-16}$, which is at the level of floating-point roundoff,
whereas the largest eigenvalue is $3.0812$.  The first nonzero eigenvalue is
\begin{align}
\lambda_2
=
1.0935\times10^{-2},
\end{align}
so that the expected zero mode is clearly separated from the positive
spectrum.  These calculations verify that the assembled numerical operator realizes the
symmetry, positive-semidefiniteness, and coefficient-space kernel identities
built into the construction.

Refinement consistency is tested on a smooth reference function using
successive uniform levels.  This dedicated refinement-consistency diagnostic
uses $\gamma_{\mathrm r}=2$.  The value is chosen specifically for this
interlevel-consistency test and is larger than the refinement-decay values
used in the denoising parameter sweep.  The experiment therefore demonstrates
the refinement behavior of this parameterized operator family and should not
be interpreted as evidence that every choice of
$\gamma_{\mathrm r}$ satisfies
\eqref{eq:refinement_consistency_limit}.

In the orthonormal implementation the normalized coefficient defect coincides
with the corresponding represented-function defect associated with
\eqref{eq:refinement_commutator_defect}.  For the state-dependent operator,
the measured sequence is
\begin{align}
6.21\times10^{-3},
\quad
4.74\times10^{-3},
\quad
9.43\times10^{-4},
\quad
1.39\times10^{-4},
\quad
1.77\times10^{-5},
\end{align}
for coarse levels $2,\ldots,6$, respectively.  The defect therefore decreases
by more than two orders of magnitude over the tested refinement sequence.
The state-independent comparison displays a small coarse-level
nonmonotonicity but converges to essentially the same fine-level scale,
$1.77\times10^{-5}$.

\begin{figure}[t]
\centering
\includegraphics[
    width=0.62\textwidth
]{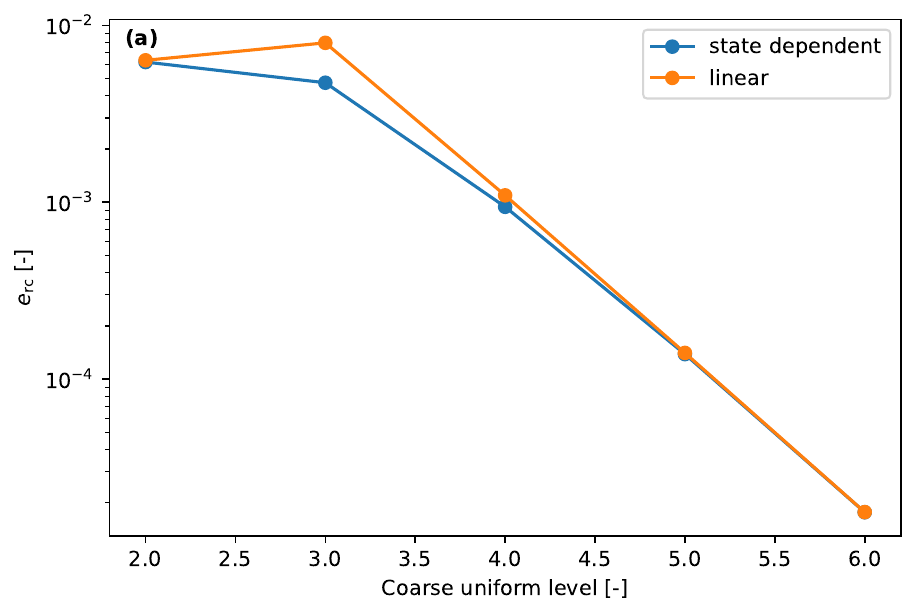}
\caption{
Refinement-commutator defect for the state-dependent and state-independent
intrinsic operators.  The defect is evaluated between successive uniform
levels using the orthonormal Haar realization of
\eqref{eq:refinement_commutator_defect}. For $\gamma_{\mathrm r}=2$, the decreasing fine-level defect provides
numerical evidence for refinement consistency of this parameterized operator
family; it is not used as a proof of the limit in
\eqref{eq:refinement_consistency_limit}.
}
\label{fig:refinement_commutator}
\end{figure}

\subsection{Discrete variational stability}
\label{subsec:numerical_energy}

The discrete energy inequality is tested using a one-dimensional
$128$-point diagnostic with noise level $\sigma=0.05$, random seed $7$,
refinement-decay parameter $\gamma_{\mathrm r}=0.4$, and a uniform
level-three Haar tree.  The tree remains fixed throughout the eight-step
diagnostic, so the experiment isolates the variational coefficient update.
In addition to monitoring
$\mathcal E(\mathfrak V^n)$, we evaluate the residual
\begin{align}
R_n
=
&
\mathcal E(\mathfrak V^{n+1})
+
\frac{1}{2h_n}
\left\|
\widetilde u^{n+1}-u^n
\right\|_H^2
\nonumber
\\
&
+
\frac{\zeta}{2}
d_{\mathbb T}^{2}
\left(
\Lambda^{n+1},
\Lambda^n
\right)
-
\mathcal E(\mathfrak V^n).
\label{eq:numerical_energy_residual}
\end{align}
For this fixed-tree diagnostic,
$d_{\mathbb T}(\Lambda^{n+1},\Lambda^n)=0$ at every step.
According to the split energy inequality
\eqref{eq:split_energy_inequality}, the residual
$R_n$ should therefore be nonpositive. 
The energy decreases from
$6.61456\times10^{-3}$ to $6.60468\times10^{-3}$ over eight iterations.
Every computed residual is negative; the residual closest to zero is
$-1.36\times10^{-8}$ at the final iteration.  The result therefore verifies
the discrete dissipation mechanism at numerical precision.

\begin{figure}[t]
\centering
\includegraphics[
    width=0.88\textwidth
]{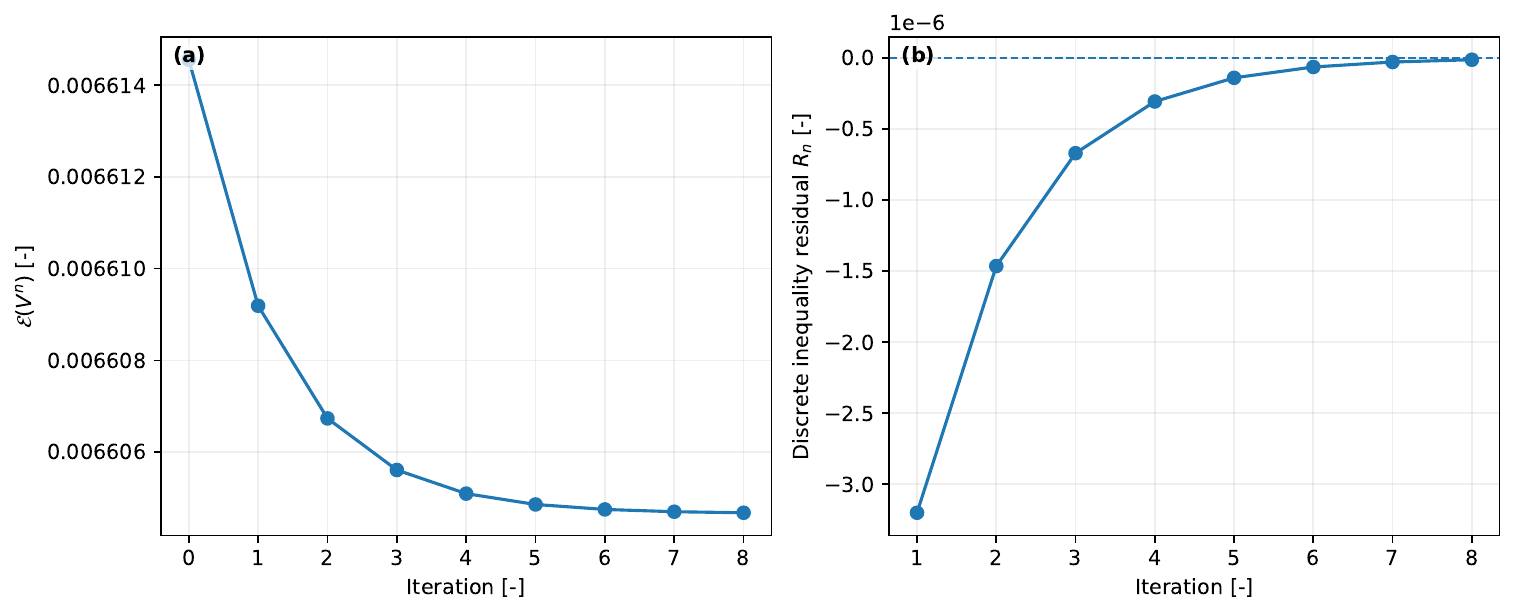}
\caption{
Discrete variational stability.
(a) State energy $\mathcal E(\mathfrak V^n)$ as a function of iteration.
(b) Residual $R_n$ defined in \eqref{eq:numerical_energy_residual}.
All nontrivial residuals remain negative and approach zero as the iteration
approaches stationarity.
}
\label{fig:energy_decay_validation}
\end{figure}

In the fully adaptive experiments below, energy monotonicity is enforced by
the frozen-weight backtracking safeguard described in
Section~\ref{sec:algorithm}.  Those calculations therefore provide numerical
stability and repeatability tests of the implemented algorithm, whereas
\cref{fig:energy_decay_validation} directly tests the exact discrete
variational inequality.

\subsection{Multiresolution localization on a heterogeneous target}
\label{subsec:numerical_localization}

A $128\times128$ synthetic target is used to examine whether the adaptive
representation distinguishes regions with qualitatively different local
structure.  The four quadrants contain, respectively, a constant field, a
smooth gradient, a sharp interface, and an oscillatory texture.  Starting
from the full multiresolution representation, the adaptive thresholding
procedure retains only
\begin{align}
N_{\Lambda}=436
\end{align}
basis functions out of $16384$, corresponding to
\begin{align}
\mathcal C_{\mathrm{rel}}
=
2.661\times10^{-2}.
\end{align}
Despite this $97.34\%$ reduction in active degrees of freedom, the adaptive
projection gives
\begin{align}
\operatorname{RMSE}
&=0.05854,
&
\operatorname{PSNR}
&=24.65,
&
\operatorname{SSIM}
&=0.8458.
\end{align}

The refinement map in \cref{fig:synthetic_adaptivity} shows the expected
spatial selectivity.  The constant region remains coarse, the smooth gradient
requires moderate resolution in the Haar representation, and the strongest
refinement is concentrated around the discontinuity and within the
oscillatory region.  The resulting error is likewise localized rather than
distributed uniformly over the domain.

\begin{figure}[t]
\centering
\includegraphics[
    width=0.98\textwidth
]{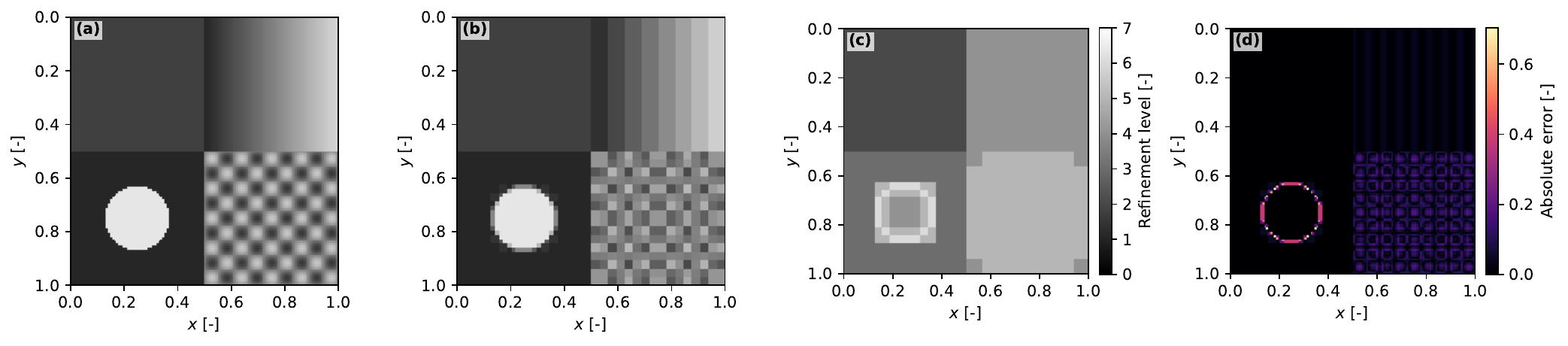}
\caption{
Haar multiresolution localization on the four-region synthetic target.
(a) Reference field.
(b) Adaptive projection.
(c) Local refinement level.
(d) Absolute reconstruction error.
Only $436$ of $16384$ full-resolution degrees of freedom are active,
corresponding to $\mathcal C_{\mathrm{rel}}=0.0266$.
}
\label{fig:synthetic_adaptivity}
\end{figure}

\subsection{Denoising accuracy versus active representation size}
\label{subsec:numerical_benchmark}

The denoising benchmark is evaluated as a multiobjective problem rather than
as a single full-resolution RMSE comparison.  The AMDI parameter sweep
produces a family of solutions spanning different values of
$\mathcal C_{\mathrm{rel}}$, which measures the active AMDI representation
relative to uniform refinement.  The conventional denoising baselines are
shown at $\mathcal C_{\mathrm{rel}}=1$ as full-problem reference methods.
For the wavelet-soft-thresholding baseline, this plotting convention refers
to the full wavelet representation on which thresholding is performed and
should not be interpreted as the fraction of nonzero coefficients remaining
after thresholding.  The resulting comparison is shown in
\cref{fig:quality_complexity}.

Total variation is the strongest conventional baseline in both independently
optimized comparisons.  Its minimum calibration RMSE is
\begin{align}
\operatorname{RMSE}_{\mathrm{TV}}
=
0.04064,
\end{align}
and its maximum calibration SSIM is
\begin{align}
\operatorname{SSIM}_{\mathrm{TV}}
=
0.87686.
\end{align}
AMDI does not improve upon these full-space extrema.  Its minimum RMSE is
$0.05134$ at
$\mathcal C_{\mathrm{rel}}=0.2017$, whereas its maximum SSIM is
\begin{align}
\operatorname{SSIM}_{\mathrm{AMDI}}
=
0.87089
\qquad
\text{at}
\qquad
\mathcal C_{\mathrm{rel}}
=
0.09326.
\end{align}
The SSIM-optimal calibration point occurs at $\beta=0$.  It therefore lies
outside the sufficient coercivity hypothesis $\beta>0$ of
Theorem~\ref{thm:inverse_reconstruction_existence}; that theorem is not
invoked to establish existence of this particular numerical operating point.

For this calibration run, the full frozen-weight proposal is accepted without
backtracking in one of the three iterations, while the remaining two updates
require safeguard backtracking.

Thus the best-SSIM AMDI solution differs from the TV optimum by only
$5.97\times10^{-3}$ in SSIM while retaining $9.33\%$ of the full
representation, i.e., $382$ active degrees of freedom instead of $4096$.
The corresponding reduction in active representation size is $90.67\%$.

The RMSE-optimal AMDI point uses $826$ active degrees of freedom.
A more strongly compressed operating point selected under the constraint
$\mathcal C_{\mathrm{rel}}\le0.1$ uses $361$ active coefficients,
\begin{align}
\mathcal C_{\mathrm{rel}}
&=0.08813,
&
\operatorname{RMSE}
&=0.05312,
&
\operatorname{SSIM}
&=0.85330.
\end{align}
These results demonstrate that the adaptive formulation can retain high
reconstruction quality while substantially reducing the active
representation.  AMDI does not minimize the reconstruction error at
unrestricted representation size; rather, its parameter sweep produces a
family of reconstruction--representation operating points associated with
different adaptive states.

\begin{table}[t]
\centering
\caption{
Principal calibration results for the denoising benchmark.
The TV entries correspond to independently RMSE- and SSIM-tuned
full-resolution calculations.
}
\label{tab:quality_complexity}
\begin{tabular}{lccc}
\hline
Method / operating point
&
RMSE
&
SSIM
&
$\mathcal C_{\mathrm{rel}}$
\\
\hline
TV, RMSE optimum
&
$0.04064$
&
-- 
&
$1.000$
\\
TV, SSIM optimum
&
--
&
$0.87686$
&
$1.000$
\\
AMDI, RMSE optimum
&
$0.05134$
&
--
&
$0.2017$
\\
AMDI, SSIM optimum
&
--
&
$0.87089$
&
$0.09326$
\\
AMDI, $\mathcal C_{\mathrm{rel}}\le0.1$
&
$0.05312$
&
$0.85330$
&
$0.08813$
\\
\hline
\end{tabular}
\end{table}

\begin{figure}[t]
\centering
\includegraphics[
    width=0.90\textwidth
]{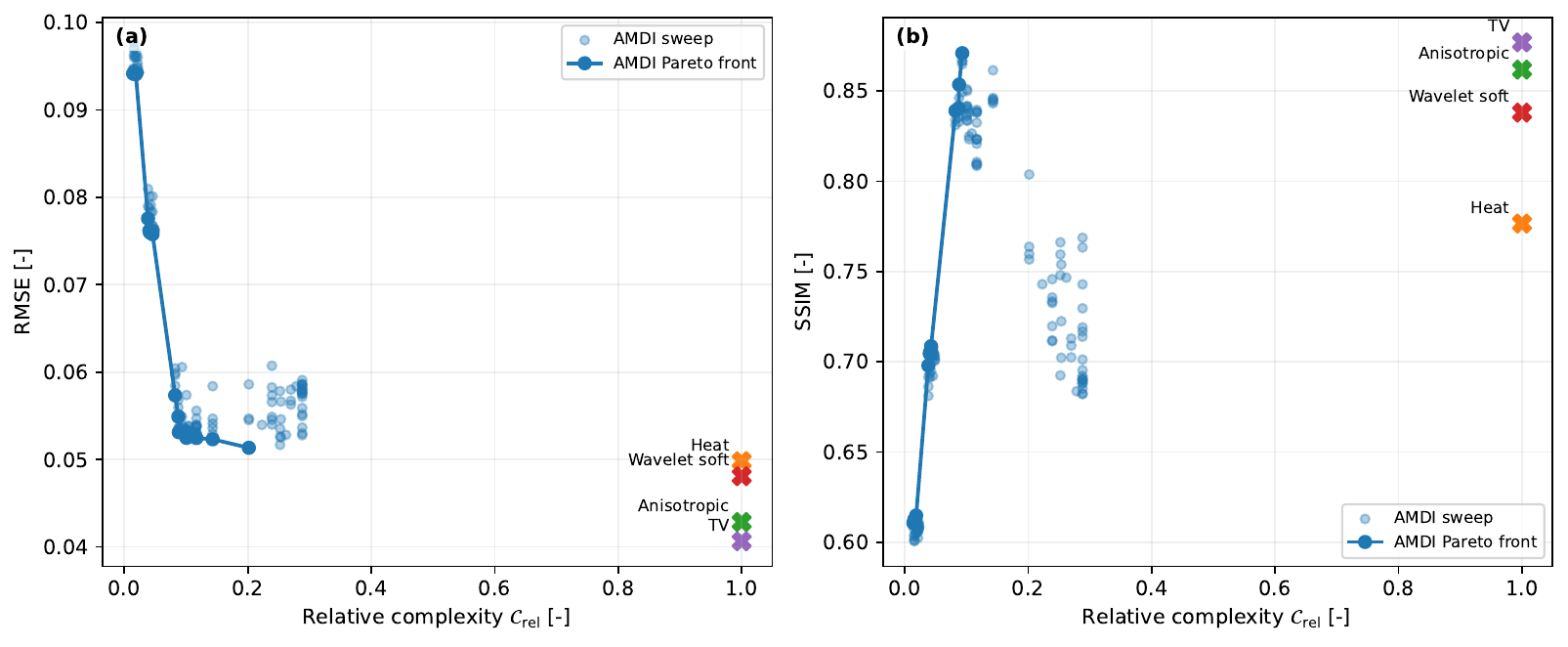}
\caption{
Quality--representation benchmark for Gaussian denoising with
$\sigma=0.08$.
(a) RMSE versus relative AMDI representation size.
(b) SSIM versus relative AMDI representation size.
The AMDI sweep generates a family of adaptive representations with varying
$\mathcal C_{\mathrm{rel}}$.  The conventional denoising methods are shown at
$\mathcal C_{\mathrm{rel}}=1$ as full-problem reference baselines; for
wavelet soft thresholding this plotting convention does not represent the
fraction of nonzero coefficients after thresholding.
Baseline parameters are optimized independently for the metric displayed in
each panel.
}
\label{fig:quality_complexity}
\end{figure}

The comparison therefore should not be interpreted as evidence that AMDI
universally outperforms TV in raw denoising accuracy.  Rather, the numerical
result is that structural reconstruction quality close to that of the best
full-space baseline can be retained with approximately one tenth of the active
representation.

\subsection{Held-out noise-realization assessment}
\label{subsec:numerical_holdout}

The preceding benchmark uses noise seed $11$ for calibration and parameter
selection.  To assess behavior after parameter selection, all method
parameters are subsequently frozen and evaluated on eight previously unseen
noise realizations with seeds
$101$, $103$, $107$, $109$, $113$, $127$, $131$, and $137$.

For the compressed AMDI operating point, the held-out calculations give
\begin{align}
\operatorname{RMSE}
&=
0.05338\pm0.00094,
\\
\operatorname{SSIM}
&=
0.85974\pm0.00657,
\\
\mathcal C_{\mathrm{rel}}
&=
0.0869\pm0.0026.
\end{align}
The SSIM-tuned full-resolution TV baseline gives
\begin{align}
\operatorname{RMSE}
&=
0.04307\pm0.00083,
\\
\operatorname{SSIM}
&=
0.87908\pm0.00565.
\end{align}
Thus the held-out SSIM gap between compressed AMDI and the full-resolution
TV optimum is approximately $1.93\times10^{-2}$, larger than on the
calibration realization.  The two Perona--Malik parameterizations selected
during calibration give held-out mean SSIM values of $0.86691$ and
$0.86252$, respectively.  The held-out results therefore confirm that the
compressed AMDI operating point remains stable across unseen noise
realizations, while also showing that its mean structural accuracy remains
below the strongest full-resolution and anisotropic-diffusion baselines on
this particular denoising problem.

\subsection{Convergence robustness and standardized initialization}
\label{subsec:numerical_robustness}

Robustness is evaluated around the compressed operating point by perturbing
one parameter at a time.  The time step $h$, diffusion weight $\alpha$,
sparsity weight $\beta$, tree penalty $\tau$, and initial coefficient
threshold are each multiplied by $1/2$ and $2$.  The experiment therefore
contains eleven parameter cases, including the unperturbed calculation.

A separate repeatability calculation uses noise seeds
$11$, $17$, $23$, $31$, and $43$.  Since seed $11$ is also the calibration
realization, these five calculations are not treated as a held-out
generalization test; that role is assigned exclusively to
\cref{subsec:numerical_holdout}.

All sixteen calculations remain energy monotone.  For these robustness
trajectories, the full frozen-weight update is accepted without backtracking
at every iteration.  Fourteen of the sixteen runs stop before the maximum
iteration count and satisfy both final convergence tolerances.  Across the
complete $192$-run calibration sweep, however, the full-step acceptance rate
has mean $0.891$ and minimum $0.167$, and $34$ runs require at least one
backtracking reduction.  The safeguard is therefore an active component of
the numerical implementation rather than a purely formal precaution.

The five-realization repeatability calculation gives
\begin{align}
\operatorname{RMSE}
&=
0.05363
\pm
0.00087,
\\
\operatorname{SSIM}
&=
0.85208
\pm
0.00503,
\\
\mathcal C_{\mathrm{rel}}
&=
0.08667
\pm
0.00254.
\label{eq:noise_robustness_results}
\end{align}
These statistics show good repeatability of the selected operating point
over the tested realizations.  The parameter perturbations reveal a more
differentiated sensitivity: changing $h$ has negligible effect on SSIM and
changing $\alpha$ produces only a modest variation, whereas doubling
$\beta$, doubling $\tau$, and doubling the raw initialization threshold
reduce the SSIM to approximately $0.830$, $0.794$, and $0.649$,
respectively.  This stronger dependence on the tree penalty, sparsity
regularization, and raw initialization threshold motivates the standardized
initialization procedure introduced below.

\begin{figure}[t]
\centering
\includegraphics[
    width=0.98\textwidth
]{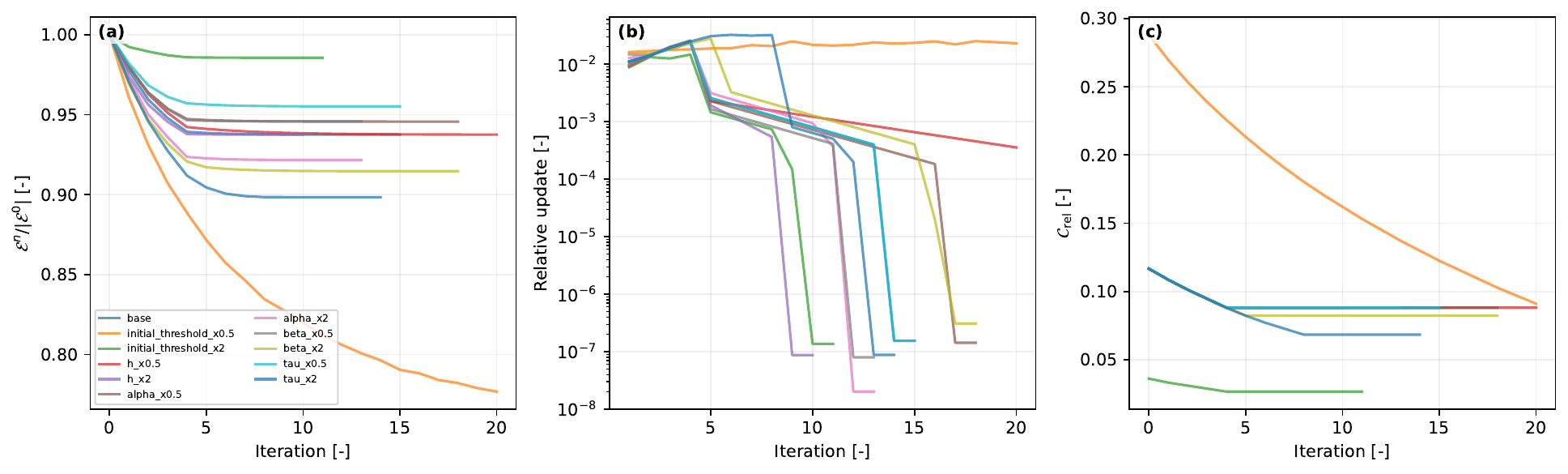}
\caption{
Convergence trajectories under one-at-a-time parameter perturbations.
(a) Energy normalized by the magnitude of the initial energy.
(b) Relative iterate update.
(c) Relative adaptive representation size.
All tested trajectories are energy monotone and the full frozen-weight update
is accepted without backtracking in these robustness runs.
Each normalized energy trajectory is used only to assess within-run
monotonicity; objective values corresponding to different parameter settings
are not compared quantitatively.
}
\label{fig:robustness_convergence}
\end{figure}

The largest sensitivity in the raw perturbation experiment originates from
the initial coefficient threshold.  This is expected because a fixed
coefficient threshold is scale dependent and can generate substantially
different initial trees.  We therefore replace it by a standardized
initialization rule based directly on the prescribed initial relative
complexity
\begin{align}
\mathcal C_{\mathrm{rel}}^0
=
\frac{N_{\Lambda^0}}{N_{\mathrm{full}}}.
\label{eq:initial_relative_complexity}
\end{align}
For each requested value of $\mathcal C_{\mathrm{rel}}^0$, the coefficient
threshold is determined automatically so that an admissible initial tree
satisfies the target complexity as closely as permitted by the discrete tree
structure.  The procedure uses only the noisy input and does not use the
reference solution.

Three initialization budgets,
$\mathcal C_{\mathrm{rel}}^0=0.10$, $0.12$, and $0.14$, are tested over five
noise realizations.  The results are summarized in
\cref{tab:standardized_initialization}.  The intermediate budget
$\mathcal C_{\mathrm{rel}}^0=0.12$ gives the best observed compromise,
\begin{align}
\operatorname{RMSE}
&=
0.05283
\pm
0.00051,
\\
\operatorname{SSIM}
&=
0.86022
\pm
0.00518,
\\
\mathcal C_{\mathrm{rel}}^{\mathrm{final}}
&=
0.08975
\pm
0.00256.
\end{align}
All fifteen standardized-initialization runs remain energy monotone and all
frozen-weight steps are accepted by the safeguard.  This complexity-based
initialization is therefore used as the preferred reproducible specification
of the initial AMDI state.

\begin{table}[t]
\centering
\caption{
Sensitivity to the prescribed initial relative-complexity budget.
Values are means and sample standard deviations over five noise
realizations.
}
\label{tab:standardized_initialization}
\begin{tabular}{cccc}
\hline
$\mathcal C_{\mathrm{rel}}^0$
&
RMSE
&
SSIM
&
$\mathcal C_{\mathrm{rel}}^{\mathrm{final}}$
\\
\hline
$0.10$
&
$0.05927\pm0.00104$
&
$0.82367\pm0.01029$
&
$0.07656\pm0.00337$
\\
$0.12$
&
$0.05283\pm0.00051$
&
$0.86022\pm0.00518$
&
$0.08975\pm0.00256$
\\
$0.14$
&
$0.05295\pm0.00069$
&
$0.85285\pm0.00724$
&
$0.09194\pm0.00222$
\\
\hline
\end{tabular}
\end{table}

The results do not imply complete independence from initialization:
different initial complexity budgets lead to measurably different
accuracy--compression balances.  They do show, however, that an arbitrary
raw threshold can be replaced by an interpretable and reproducible
dimensionless initialization parameter.

\subsection{Matched-complexity ablation}
\label{subsec:numerical_ablation}

The ablation analysis is performed at matched complexity.  For every noise
realization, the full AMDI evolution first selects an adaptive tree.  The
operator variants are then evaluated on exactly that same tree, so that
differences in the reconstruction metrics cannot be attributed to different
numbers of active coefficients.  The maximum within-seed complexity mismatch
is zero.

The results over five noise realizations are listed in
\cref{tab:matched_ablation}.  Freezing the already selected tree has
essentially no effect on the reconstruction, as expected at the final
operating point.  Replacing the state-dependent interaction weights by
state-independent weights increases the mean RMSE from $0.05362$ to
$0.06046$, a deterioration of approximately $12.8\%$ at identical
complexity. The SSIM is also slightly higher for the state-independent variant in all
five paired realizations. The state dependence therefore provides
a clear improvement in amplitude accuracy as measured by RMSE, but the data
do not support a claim that it improves every reconstruction metric.

Removing the intrinsic diffusion contribution after the adaptive tree has
already been selected gives a slightly lower RMSE and slightly higher SSIM
in all five paired realizations.  This result concerns the coefficient
update on a fixed AMDI-selected representation.  It does not constitute an
ablation of intrinsic diffusion from the preceding adaptive tree evolution.
For the present identity-denoising problem, the fixed-tree comparison
therefore provides no evidence of an additional denoising advantage from the
diffusion contribution itself, whereas the state-dependent interaction
weights produce a clear reduction in RMSE at identical representation size.

\begin{table}[t]
\centering
\caption{
Matched-complexity ablation over five noise realizations.
Every variant is evaluated on the same AMDI-selected tree within each seed.
}
\label{tab:matched_ablation}
\begin{tabular}{lccc}
\hline
Variant
&
RMSE
&
SSIM
&
$\mathcal C_{\mathrm{rel}}$
\\
\hline
Adaptive AMDI
&
$0.05362\pm0.00087$
&
$0.85220\pm0.00506$
&
$0.08667\pm0.00254$
\\
Fixed selected tree
&
$0.05361\pm0.00087$
&
$0.85222\pm0.00505$
&
$0.08667\pm0.00254$
\\
No intrinsic diffusion
&
$0.05323\pm0.00084$
&
$0.85543\pm0.00544$
&
$0.08667\pm0.00254$
\\
State-independent weights
&
$0.06046\pm0.00127$
&
$0.85594\pm0.00495$
&
$0.08667\pm0.00254$
\\
\hline
\end{tabular}
\end{table}

\begin{figure}[t]
\centering
\includegraphics[
    width=0.88\textwidth
]{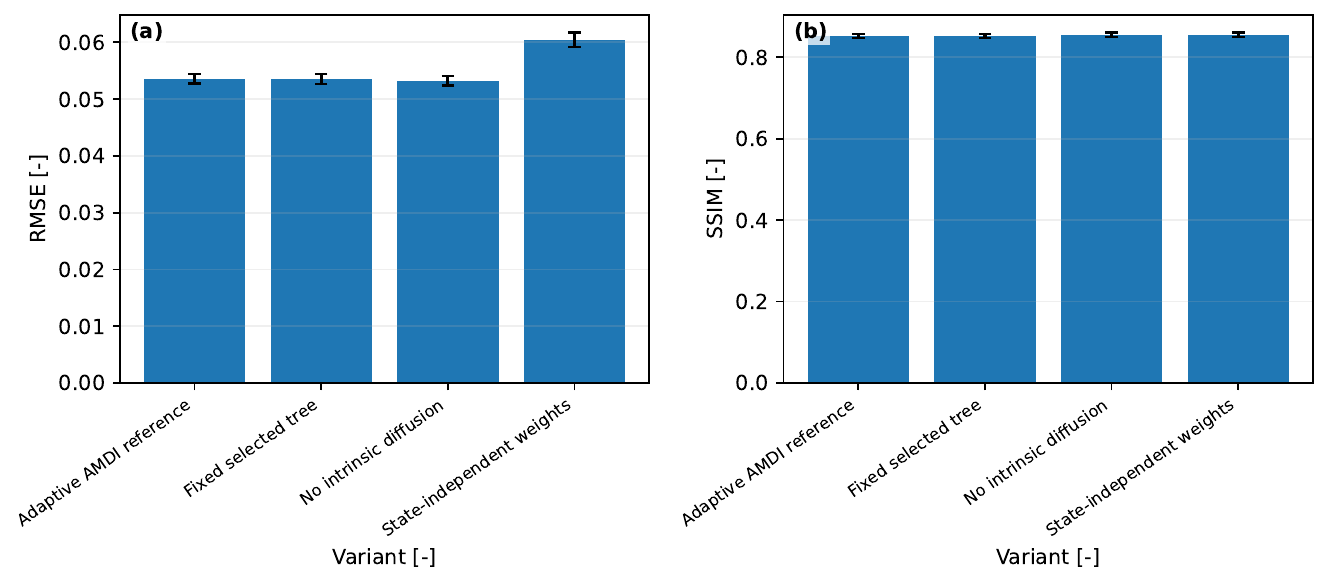}
\caption{
Matched-complexity ablation.
(a) RMSE and
(b) SSIM, reported as mean $\pm$ standard deviation over five noise
realizations.
All variants use exactly the same selected tree within each realization.
}
\label{fig:matched_complexity_ablation}
\end{figure}

\subsection{Higher-order multiwavelet cross-check with VAMPyR}
\label{subsec:numerical_vampyr}

The coefficient-level AMDI calculations above use an orthonormal Haar
backend.  To determine whether the regional localization observed in the
separate Haar approximation diagnostic is also recovered by a higher-order
adaptive multiresolution representation, we perform a cross-representation
calculation with VAMPyR, the Python interface to the MRCPP multiwavelet
framework \cite{BjorgveEtAl2024}.  This experiment cross-checks the adaptive
representation rather than independently validating the AMDI diffusion
operator.

For the four-region target of
\cref{subsec:numerical_localization}, the depth-capped VAMPyR calculation
contains $394$ adaptive end nodes.  Of these, $240$ lie at the deepest
available scale, so the imposed depth cap is active in this localization
experiment.  Consequently, the regional node counts and effective refinement
levels are not interpreted as converged quantities; the calculation is used
only to identify which regions attract adaptive resolution.

Their regional distribution is highly nonuniform:
\begin{align}
N_{\mathrm{end}}^{\mathrm{constant}}
&=1,
&
N_{\mathrm{end}}^{\mathrm{gradient}}
&=1,
&
N_{\mathrm{end}}^{\mathrm{edge}}
&=328,
&
N_{\mathrm{end}}^{\mathrm{texture}}
&=64.
\end{align}
Thus approximately $83.25\%$ of the end nodes are concentrated in the
sharp-interface region and $16.24\%$ in the oscillatory region; the constant
and smooth-gradient regions each require only one end node.

Both representations identify the sharp-interface and oscillatory regions as
substantially more demanding than the smooth regions.  Their detailed
allocation of resolution is not identical, and the VAMPyR edge-region
resolution is additionally limited by the imposed depth cap.  Absolute
refinement levels should therefore not be compared directly.  The Haar
realization is piecewise constant and has approximation order one, whereas
the VAMPyR calculation uses order-five multiwavelets.  We interpret the
comparison only as a qualitative regional localization cross-check, not as
quantitative agreement of refinement levels.

\begin{figure}[t]
\centering
\includegraphics[
    width=0.96\textwidth
]{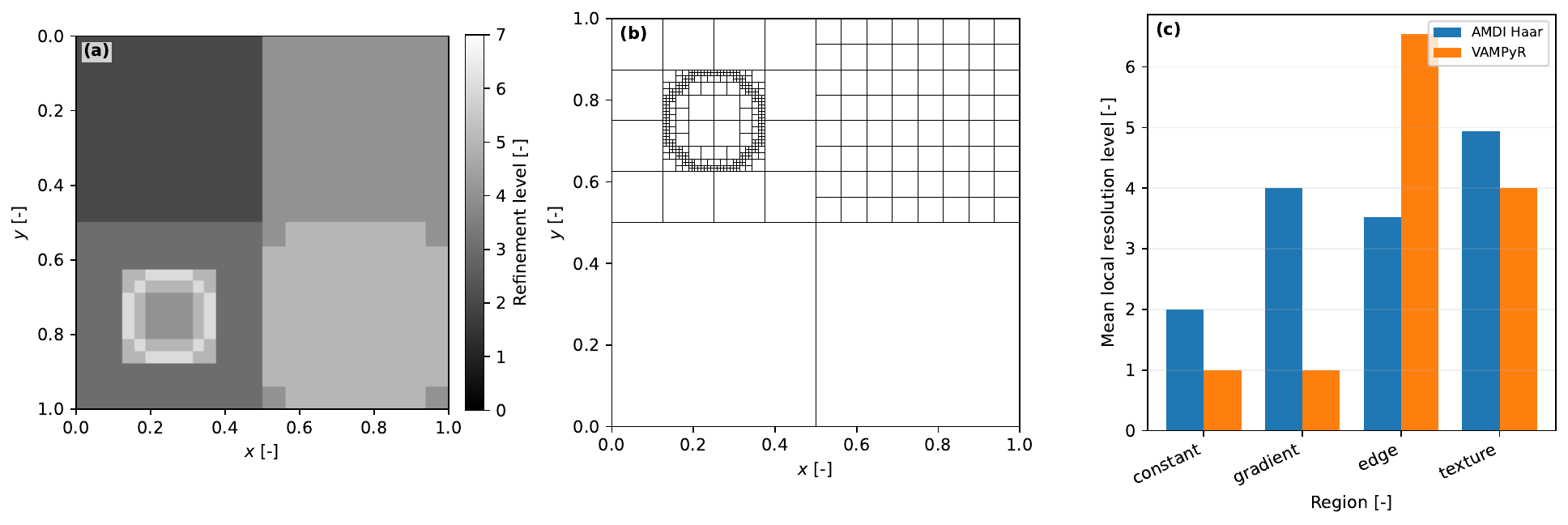}
\caption{
Cross-representation multiresolution localization comparison.
(a) Haar refinement map.
(b) VAMPyR adaptive end-node grid.
(c) Regional resolution statistics.
Both representations allocate substantially greater resolution to the sharp
and oscillatory regions than to the smooth regions.  Quantitative refinement
levels are not compared because the two representations have different local
approximation orders and the VAMPyR localization calculation reaches the
imposed depth cap.
}
\label{fig:vampyr_localization}
\end{figure}

A second VAMPyR calculation tests approximation precision independently of
the localization experiment.  The target is a smooth multiscale function
containing a ramp, a localized Gaussian contribution, a steep hyperbolic
tangent transition, and a localized oscillation.  An order-five
multiwavelet representation is used with maximum depth eight.  The
$10^{-5}$ projection is taken as the internal reference and contains
$367$ end nodes at depth six; the imposed depth cap is never reached.

As the requested precision is tightened from $10^{-2}$ to $3\times10^{-5}$,
the number of end nodes increases from $28$ to $328$, while the $L^2$
distance to the reference decreases from
\begin{align}
7.06\times10^{-4}
\quad\text{to}\quad
1.57\times10^{-6}.
\end{align}
This corresponds to a reduction of the distance to the reference projection
by a factor of approximately $451$ before the reference level itself is
reached.  The experiment is a self-consistency audit of the adaptive
multiwavelet representation rather than a convergence study against an exact
analytic solution.  The relevant observation is that adaptive complexity is
nondecreasing and the distance to the tighter reference projection is
nonincreasing at every tested precision, with no calculation reaching the
maximum allowed depth.

\begin{figure}[t]
\centering
\includegraphics[
    width=0.82\textwidth
]{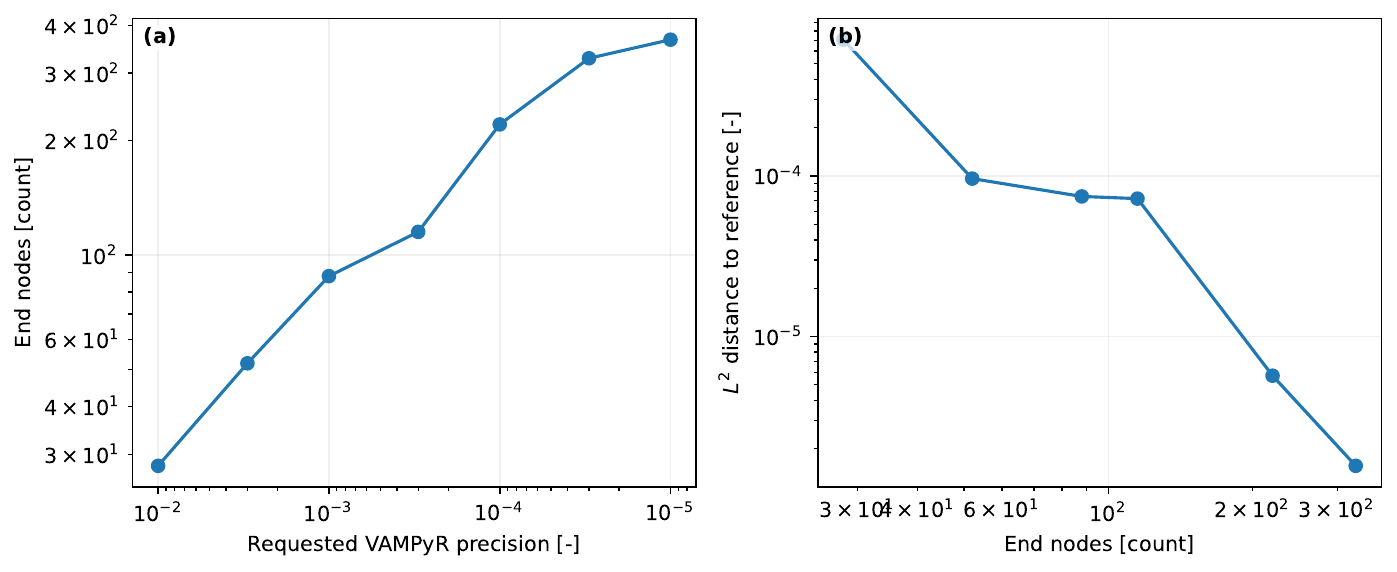}
\caption{
VAMPyR precision audit for the smooth multiscale target.
(a) Adaptive end-node count as the requested precision is tightened.
(b) $L^2$ distance to the $10^{-5}$ reference projection as a function of
adaptive complexity.
The reference point itself has zero distance and is not used to define a
finite error-reduction factor or a convergence order.
}
\label{fig:vampyr_precision}
\end{figure}

\section{Conclusions}
\label{sec:conclusions}

We have developed a variational framework in which an adaptive
multiresolution state generates its own intrinsic diffusion operator.  The
state combines an admissible multiresolution tree, the associated active
approximation space, and its coefficient representation.  Symmetric
nonnegative interactions on this hierarchy define a positive semidefinite
diffusion operator with the expected coefficient-space null mode.  Because the
basis is not assumed to be orthonormal, the corresponding function-space
operator is formulated through the Gram matrix, which preserves the
variational structure for general Riesz bases.

The coupled evolution of the represented function and the adaptive tree is
defined by a time-discrete variational principle.  This avoids
introducing a formal differential structure on the discrete refinement
hierarchy and yields existence of each adaptive update together with a
discrete energy inequality.  The resulting Adaptive Multiresolution Diffusion
Imaging (AMDI) formulation combines data fidelity, intrinsic multiresolution
diffusion, coefficient sparsity, and tree complexity within a single
state-dependent energy.  A split implementation retains the variational
structure, while a frozen-weight approximation provides an efficient
symmetric positive-definite coefficient update supplemented by an explicit
energy safeguard.

The numerical results support the principal structural and algorithmic
properties of the construction.  The assembled intrinsic operator satisfies
the expected symmetry, positive-semidefiniteness, and coefficient-space
kernel identities, while the refinement study examines interlevel consistency
for a specified refinement-decay parameter and the exact variational
diagnostic verifies discrete energy dissipation.  A separate Haar
multiresolution experiment demonstrates strong localization and compression
on heterogeneous data.  In the denoising benchmark, AMDI does not outperform
full-resolution total-variation denoising in unrestricted accuracy, but it
retains high structural fidelity with a small fraction of the full active
representation.  Held-out noise-realization tests confirm that this
reconstruction--representation balance persists across unseen realizations,
while also showing a larger accuracy gap to the strongest baselines than on
the calibration realization.  Energy monotonicity is maintained by the
implemented backtracking safeguard, although reconstruction quality is more
sensitive to the sparsity weight, tree penalty, and raw initialization
threshold than to the time step or diffusion weight.  The complexity-based
initialization removes the strongest dependence on an arbitrary
scale-dependent threshold.  Matched-representation ablations show that
state-dependent interactions consistently reduce RMSE relative to
state-independent weights, while the fixed-tree experiment does not establish
an additional denoising benefit from the intrinsic diffusion term itself.
Finally, a higher-order VAMPyR/MRCPP cross-check recovers the same qualitative
regional localization of multiresolution effort, and the separate precision
audit exhibits monotone self-consistency as the requested multiwavelet
precision is tightened.

The present work establishes a time-discrete adaptive variational theory and
its numerical realization.  It does not yet provide a continuous-time limit
for the coupled state evolution, nor does it establish asymptotic
approximation rates or optimal-complexity results.  These questions, together
with extensions to more general inverse problems and higher-order
multiwavelet backends, provide natural directions for future analysis.

\section*{Data Availability Statement}

The source code, numerical implementation, experiment scripts, and computational workflow required to reproduce the results reported in this study are openly available in the AMDI numerical validation repository at
\url{https://github.com/Christian48596/amdi_validation}.
All numerical data presented in the manuscript are generated computationally by the provided scripts, together with the corresponding parameter settings, software environment, validation tests, and instructions for reproducing the reported calculations. No proprietary or restricted datasets were used in this study.

\bibliographystyle{siamplain}
\bibliography{references}

@article{RudinOsherFatemi1992,
  author = {Leonid I. Rudin and Stanley Osher and Emad Fatemi},
  title = {Nonlinear Total Variation Based Noise Removal Algorithms},
  journal = {Physica D},
  year = {1992},
  volume = {60},
  number = {1--4},
  pages = {259--268},
  doi = {10.1016/0167-2789(92)90242-F}
}

@article{PeronaMalik1990,
  author = {Pietro Perona and Jitendra Malik},
  title = {Scale-Space and Edge Detection Using Anisotropic Diffusion},
  journal = {IEEE Transactions on Pattern Analysis and Machine Intelligence},
  year = {1990},
  volume = {12},
  number = {7},
  pages = {629--639},
  doi = {10.1109/34.56205}
}

@article{CoifmanMaggioni2006,
  author = {Ronald R. Coifman and Mauro Maggioni},
  title = {Diffusion Wavelets},
  journal = {Applied and Computational Harmonic Analysis},
  year = {2006},
  volume = {21},
  number = {1},
  pages = {53--94},
  doi = {10.1016/j.acha.2006.04.004}
}

@article{CohenDahmenDeVore2001,
  author = {Albert Cohen and Wolfgang Dahmen and Ronald DeVore},
  title = {Adaptive Wavelet Methods for Elliptic Operator Equations: Convergence Rates},
  journal = {Mathematics of Computation},
  year = {2001},
  volume = {70},
  number = {233},
  pages = {27--75},
  doi = {10.1090/S0025-5718-00-01252-7}
}

@article{CohenDahmenDeVore2002,
  author = {Albert Cohen and Wolfgang Dahmen and Ronald DeVore},
  title = {Adaptive Wavelet Methods II: Beyond the Elliptic Case},
  journal = {Foundations of Computational Mathematics},
  year = {2002},
  volume = {2},
  number = {3},
  pages = {203--245},
  doi = {10.1007/s102080010027}
}

@article{Stevenson2003,
  author = {Rob Stevenson},
  title = {Adaptive Solution of Operator Equations Using Wavelet Frames},
  journal = {SIAM Journal on Numerical Analysis},
  year = {2003},
  volume = {41},
  number = {3},
  pages = {1074--1100},
  doi = {10.1137/S0036142902407988}
}

@article{Dorfler1996,
  author = {Willy D{\"o}rfler},
  title = {A Convergent Adaptive Algorithm for {Poisson}'s Equation},
  journal = {SIAM Journal on Numerical Analysis},
  year = {1996},
  volume = {33},
  number = {3},
  pages = {1106--1124},
  doi = {10.1137/0733054}
}

@article{MorinNochettoSiebert2002,
  author = {Pedro Morin and Ricardo H. Nochetto and Kunibert G. Siebert},
  title = {Convergence of Adaptive Finite Element Methods},
  journal = {SIAM Review},
  year = {2002},
  volume = {44},
  number = {4},
  pages = {631--658},
  doi = {10.1137/S0036144502409093}
}

@article{BinevDahmenDeVore2004,
  author = {Peter Binev and Wolfgang Dahmen and Ronald DeVore},
  title = {Adaptive Finite Element Methods with Convergence Rates},
  journal = {Numerische Mathematik},
  year = {2004},
  volume = {97},
  pages = {219--268},
  doi = {10.1007/s00211-003-0492-7}
}

@book{AmbrosioGigliSavare2008,
  author = {Luigi Ambrosio and Nicola Gigli and Giuseppe Savar{\'e}},
  title = {Gradient Flows in Metric Spaces and in the Space of Probability Measures},
  publisher = {Birkh{\"a}user},
  address = {Basel},
  year = {2008},
  edition = {2},
  doi = {10.1007/978-3-7643-8722-8}
}

@article{Ambrosio1995,
  author = {Luigi Ambrosio},
  title = {Minimizing Movements},
  journal = {Rendiconti della Accademia Nazionale delle Scienze detta dei XL},
  year = {1995},
  volume = {19},
  pages = {191--246}
}

@article{GianazzaGobbinoSavare1994,
  author = {Ugo Gianazza and Massimo Gobbino and Giuseppe Savar{\'e}},
  title = {Evolution Problems and Minimizing Movements},
  journal = {Rendiconti Lincei - Matematica e Applicazioni},
  year = {1994},
  volume = {5},
  number = {4},
  pages = {289--296},
  note = {Available at \url{https://eudml.org/doc/244234}}
}

@book{Braides2002,
  author    = {Andrea Braides},
  title     = {{$\Gamma$}-Convergence for Beginners},
  publisher = {Oxford University Press},
  year      = {2002},
  doi       = {10.1093/acprof:oso/9780198507840.001.0001}
}

@book{Braides2014,
  author    = {Andrea Braides},
  title     = {Local Minimization, Variational Evolution and {$\Gamma$}-Convergence},
  publisher = {Springer},
  year      = {2014},
  doi       = {10.1007/978-3-319-01982-6}
}

@article{MielkeOrtiz2008,
  author = {Alexander Mielke and Michael Ortiz},
  title = {A Class of Minimum Principles for Characterizing the Trajectories and the Relaxed Solutions of Rate-Independent Systems},
  journal = {ESAIM: Control, Optimisation and Calculus of Variations},
  year = {2008},
  volume = {14},
  number = {3},
  pages = {494--516},
  doi = {10.1051/cocv:2008006}
}

@article{BelkinNiyogi2003,
  author = {Mikhail Belkin and Partha Niyogi},
  title = {{Laplacian} Eigenmaps for Dimensionality Reduction and Data Representation},
  journal = {Neural Computation},
  year = {2003},
  volume = {15},
  number = {6},
  pages = {1373--1396},
  doi = {10.1162/089976603321780317}
}

@article{BelkinNiyogi2008,
  author = {Mikhail Belkin and Partha Niyogi},
  title = {Towards a Theoretical Foundation for {Laplacian}-Based Manifold Methods},
  journal = {Journal of Computer and System Sciences},
  year = {2008},
  volume = {74},
  number = {8},
  pages = {1289--1308},
  doi = {10.1016/j.jcss.2007.08.006}
}

@article{CoifmanLafon2006,
  author = {Ronald R. Coifman and St{\'e}phane Lafon},
  title = {Diffusion Maps},
  journal = {Applied and Computational Harmonic Analysis},
  year = {2006},
  volume = {21},
  number = {1},
  pages = {5--30},
  doi = {10.1016/j.acha.2006.04.006}
}

@article{CoifmanEtAl2005,
  author = {Ronald R. Coifman and St{\'e}phane Lafon and Ann B. Lee and Mauro Maggioni and Boaz Nadler and Frederick Warner and Steven W. Zucker},
  title = {Geometric Diffusions as a Tool for Harmonic Analysis and Structure Definition of Data: Diffusion Maps},
  journal = {Proceedings of the National Academy of Sciences},
  year = {2005},
  volume = {102},
  number = {21},
  pages = {7426--7431},
  doi = {10.1073/pnas.0500334102}
}

@article{Singer2006,
  author = {Amit Singer},
  title = {From Graph to Manifold {Laplacian}: The Convergence Rate},
  journal = {Applied and Computational Harmonic Analysis},
  year = {2006},
  volume = {21},
  number = {1},
  pages = {128--134},
  doi = {10.1016/j.acha.2006.03.004}
}

@article{SingerWu2012,
  author = {Amit Singer and Hau-Tieng Wu},
  title = {Vector Diffusion Maps and the Connection {Laplacian}},
  journal = {Communications on Pure and Applied Mathematics},
  year = {2012},
  volume = {65},
  number = {8},
  pages = {1067--1144},
  doi = {10.1002/cpa.21395}
}

@book{Chung1997,
  author = {Fan R. K. Chung},
  title = {Spectral Graph Theory},
  publisher = {American Mathematical Society},
  year = {1997},
  volume = {92},
  series = {CBMS Regional Conference Series in Mathematics}
}

@article{DonohoJohnstone1994,
  author = {David L. Donoho and Iain M. Johnstone},
  title = {Ideal Spatial Adaptation by Wavelet Shrinkage},
  journal = {Biometrika},
  year = {1994},
  volume = {81},
  number = {3},
  pages = {425--455},
  doi = {10.1093/biomet/81.3.425}
}

@article{Donoho1995,
  author = {David L. Donoho},
  title = {Nonlinear Solution of Linear Inverse Problems by Wavelet--Vaguelette Decomposition},
  journal = {Applied and Computational Harmonic Analysis},
  year = {1995},
  volume = {2},
  number = {2},
  pages = {101--126},
  doi = {10.1006/acha.1995.1008}
}

@article{FredianiFossgaardFlaRuud2013,
  author = {Luca Frediani and Eirik Fossgaard and Tor Fl{\aa} and Kenneth Ruud},
  title = {Fully Adaptive Algorithms for Multivariate Integral Equations Using the Non-Standard Form and Multiwavelets with Applications to the {Poisson} and Bound-State {Helmholtz} Kernels in Three Dimensions},
  journal = {Molecular Physics},
  year = {2013},
  volume = {111},
  number = {9-11},
  pages = {1143--1160},
  doi = {10.1080/00268976.2013.810793}
}

@article{BeylkinMonzon2005,
  author = {Gregory Beylkin and Lucas Monz{\'o}n},
  title = {On Approximation of Functions by Exponential Sums},
  journal = {Applied and Computational Harmonic Analysis},
  year = {2005},
  volume = {19},
  number = {1},
  pages = {17--48},
  doi = {10.1016/j.acha.2005.01.003}
}

@article{HarrisonEtAl2004,
  author = {Robert J. Harrison and Gregory I. Fann and Takeshi Yanai and Zhengting Gan and Gregory Beylkin},
  title = {Multiresolution Quantum Chemistry in Multiwavelet Bases},
  journal = {Journal of Chemical Physics},
  year = {2004},
  volume = {121},
  pages = {11587--11598},
  doi = {10.1063/1.1818731}
}

@article{FannEtAl2004,
  author = {Gregory I. Fann and Robert J. Harrison and Gregory Beylkin},
  title = {Multiresolution Quantum Chemistry: Basic Theory and Initial Applications},
  journal = {Journal of Chemical Physics},
  year = {2004},
  volume = {121},
  pages = {11505--11512},
  doi = {10.1063/1.1818730}
}

@article{WindEtAl2023,
  author = {Peter Wind and Magnar Bj{\o}rgve and Anders Brakestad and Gabriel A. Gerez Sazo and Stig Rune Jensen and Roberto Di Remigio and Luca Frediani},
  title = {MRChem: Multiresolution Analysis Code for Molecular Electronic Structure Calculations},
  journal = {Journal of Chemical Theory and Computation},
  year = {2023},
  volume = {19},
  pages = {137--146},
  doi = {10.1021/acs.jctc.2c00982}
}

@article{BjorgveEtAl2024,
  author = {Magnar Bj{\o}rgve and Christian Tantardini and Stig Rune Jensen and Gabriel Gerez Sazo and Peter Wind and Roberto Di Remigio and Evgueni Dinvay and Luca Frediani},
  title = {{VAMPyR}: A High-Level {Python} Library for Mathematical Operations in a Multiwavelet Representation},
  journal = {Journal of Chemical Physics},
  year = {2024},
  doi = {10.1063/5.0195036}
}

@article{DinvayZabelinaFrediani2024,
  author = {Evgueni Dinvay and Yuliya Zabelina and Luca Frediani},
  title = {Multiresolution of the One Dimensional Free-Particle Propagator. Part 1: Construction},
  journal = {Computer Physics Communications},
  year = {2024},
  doi = {10.1016/j.cpc.2024.109462}
}

@article{HarrisonEtAlMADNESS2008,
  author = {Robert J. Harrison and Gregory I. Fann and Takeshi Yanai and Zhengting Gan and Gregory Beylkin},
  title = {{MADNESS}: A Multiresolution Adaptive Numerical Environment for Scientific Simulation},
  journal = {SIAM Journal on Scientific Computing},
  year = {2008},
  volume = {30},
  number = {1},
  pages = {235--262},
  doi = {10.1137/06066594X}
}

@article{LiEtAl2021FNO,
  author = {Zongyi Li and Nikola Kovachki and Kamyar Azizzadenesheli and Burigede Liu and Kaijun Bhattacharya and Andrew Stuart and Anima Anandkumar},
  title = {Fourier Neural Operator for Parametric Partial Differential Equations},
  journal = {International Conference on Learning Representations},
  year = {2021},
  note = {Available at \url{https://openreview.net/forum?id=c8P9NQVtmnO}}
}

@article{LuEtAl2021DeepONet,
  author = {Lu Lu and Pengzhan Jin and George Em Karniadakis},
  title = {Learning Nonlinear Operators via {DeepONet} Based on the Universal Approximation Theorem of Operators},
  journal = {Nature Machine Intelligence},
  year = {2021},
  volume = {3},
  pages = {218--229},
  doi = {10.1038/s42256-021-00302-5}
}

@article{KovachkiEtAl2023,
  author = {Nikola Kovachki and Samuel Lanthaler and Siddhartha Mishra and others},
  title = {Neural Operator: Learning Maps Between Function Spaces},
  journal = {Journal of Machine Learning Research},
  year = {2023},
  volume = {24},
  pages = {1--97}
}

@article{TripuraChakraborty2023,
  author = {Tapas Tripura and Souvik Chakraborty},
  title = {Wavelet Neural Operator: A Neural Operator for Parametric {PDEs}},
  journal = {Computer Methods in Applied Mechanics and Engineering},
  year = {2023},
  volume = {404},
  pages = {115783},
  doi = {10.1016/j.cma.2022.115783}
}

@article{LiKovachkiAzizzadenesheli2020,
  author = {Zongyi Li and Nikola Kovachki and Kamyar Azizzadenesheli and others},
  title = {Neural Operator: Graph Kernel Network for Partial Differential Equations},
  journal = {arXiv},
  year = {2020}
}

@article{GreengardRokhlin1987,
  author = {Leslie Greengard and Vladimir Rokhlin},
  title = {A Fast Algorithm for Particle Simulations},
  journal = {Journal of Computational Physics},
  year = {1987},
  volume = {73},
  pages = {325--348},
  doi = {10.1016/0021-9991(87)90140-9}
}

@article{GilboaOsher2008,
  author = {Guy Gilboa and Stanley Osher},
  title = {Nonlocal Operators with Applications to Image Processing},
  journal = {Multiscale Modeling and Simulation},
  year = {2008},
  volume = {7},
  number = {3},
  pages = {1005--1028},
  doi = {10.1137/070698592}
}

@article{BuadesCollMorel2005,
  author = {Antoni Buades and Bartomeu Coll and Jean-Michel Morel},
  title = {A Non-Local Algorithm for Image Denoising},
  journal = {CVPR},
  year = {2005}
}

@article{DeVore1998,
  author = {Ronald A. DeVore},
  title = {Nonlinear Approximation},
  journal = {Acta Numerica},
  year = {1998},
  volume = {7},
  pages = {51--150},
  doi = {10.1017/S0962492900002816}
}

@book{Temlyakov2011,
  author = {Vladimir N. Temlyakov},
  title = {Greedy Approximation},
  publisher = {Cambridge University Press},
  year = {2011},
  doi = {10.1017/CBO9781139003810}
}

@article{Dahmen1997,
  author = {Wolfgang Dahmen},
  title = {Wavelet and Multiscale Methods for Operator Equations},
  journal = {Acta Numerica},
  year = {1997},
  volume = {6},
  pages = {55--228},
  doi = {10.1017/S0962492900002725}
}

@article{Stevenson2004Compressibility,
  author = {Rob Stevenson},
  title = {On the Compressibility of Operators in Wavelet Coordinates},
  journal = {SIAM Journal on Mathematical Analysis},
  year = {2004},
  volume = {35},
  number = {5},
  pages = {1110--1132},
  doi = {10.1137/S0036141002411520}
}

@article{HackbuschBorm2002,
  author = {Wolfgang Hackbusch and Steffen B{\"o}rm},
  title = {Data-Sparse Approximation by Adaptive H$^2$-Matrices},
  journal = {Computing},
  year = {2002},
  volume = {69},
  number = {1},
  pages = {1--35},
  doi = {10.1007/s00607-002-1450-4}
}

@book{Hackbusch2017,
  author = {Wolfgang Hackbusch},
  title = {Numerical Tensor Calculus},
  publisher = {Springer},
  year = {2012},
  doi = {10.1007/978-3-642-28027-6}
}

@article{BeylkinMohlenkamp2002,
  author = {Gregory Beylkin and Martin J. Mohlenkamp},
  title = {Numerical Operator Calculus in Higher Dimensions},
  journal = {Proceedings of the National Academy of Sciences},
  year = {2002},
  volume = {99},
  number = {16},
  pages = {10246--10251},
  doi = {10.1073/pnas.152329799}
}

@article{GuptaXiaoBogdan2021,
  author = {Gaurav Gupta and Xiongye Xiao and Paul Bogdan},
  title = {Multiwavelet-Based Operator Learning for Differential Equations},
  journal = {Advances in Neural Information Processing Systems},
  year = {2021},
  volume = {34},
  pages = {24048--24062},
  note = {Available at \url{https://papers.nips.cc/paper/2021/hash/c9e5c2b59d98488fe1070e744041ea0e-Abstract.html}}
}

@article{CohenDahmenDeVore2005,
  author = {Albert Cohen and Wolfgang Dahmen and Ronald A. DeVore},
  title = {Adaptive Wavelet Schemes for Nonlinear Variational Problems},
  journal = {SIAM Journal on Numerical Analysis},
  year = {2005},
  volume = {41},
  number = {5},
  pages = {1785--1823},
  doi = {10.1137/S0036142902412269}
}

@article{DahlkeDahmenUrban2002,
  author = {Stephan Dahlke and Wolfgang Dahmen and Karsten Urban},
  title = {Adaptive Wavelet Methods for Saddle Point Problems: Convergence Rates},
  journal = {SIAM Journal on Numerical Analysis},
  year = {2002},
  volume = {40},
  number = {4},
  pages = {1230--1262}
}

@book{Cohen2003Book,
  author = {Albert Cohen},
  title = {Numerical Analysis of Wavelet Methods},
  publisher = {Elsevier},
  year = {2003}
}

@article{Dahmen2001JCAM,
  author = {Wolfgang Dahmen},
  title = {Wavelet Methods for {PDEs}: Some Recent Developments},
  journal = {Journal of Computational and Applied Mathematics},
  year = {2001},
  volume = {128},
  pages = {133--185}
}

@article{DahmenKunoth1992,
  author = {Wolfgang Dahmen and Angela Kunoth},
  title = {Multilevel Preconditioning},
  journal = {Numerische Mathematik},
  year = {1992},
  volume = {63},
  pages = {315--344}
}

@article{Beylkin1992Operators,
  author = {Gregory Beylkin},
  title = {On the Representation of Operators in Bases of Compactly Supported Wavelets},
  journal = {SIAM Journal on Numerical Analysis},
  year = {1992},
  volume = {29},
  number = {6},
  pages = {1716--1740},
  doi = {10.1137/0729097}
}

@article{BeylkinCoifmanRokhlin1991,
  author = {Gregory Beylkin and Ronald R. Coifman and Vladimir Rokhlin},
  title = {Fast Wavelet Transforms and Numerical Algorithms I},
  journal = {Communications on Pure and Applied Mathematics},
  year = {1991},
  volume = {44},
  number = {2},
  pages = {141--183},
  doi = {10.1002/cpa.3160440202}
}

@article{AlpertBeylkinCoifmanRokhlin1993,
  author = {Bradley Alpert and Gregory Beylkin and Ronald R. Coifman and Vladimir Rokhlin},
  title = {Wavelet-like Bases for the Fast Solution of Second-Kind Integral Equations},
  journal = {SIAM Journal on Scientific Computing},
  year = {1993},
  volume = {14},
  number = {1},
  pages = {159--184},
  doi = {10.1137/0914010}
}

@article{Alpert1993,
  author = {Bradley Alpert},
  title = {A Class of Bases in $L^2$ for the Sparse Representation of Integral Operators},
  journal = {SIAM Journal on Mathematical Analysis},
  year = {1993},
  volume = {24},
  number = {1},
  pages = {246--262},
  doi = {10.1137/0524016}
}

@article{BeylkinKeiser1997,
  author = {Gregory Beylkin and James M. Keiser},
  title = {On the Adaptive Numerical Solution of Nonlinear Partial Differential Equations in Wavelet Bases},
  journal = {Journal of Computational Physics},
  year = {1997},
  volume = {132},
  number = {2},
  pages = {233--259},
  doi = {10.1006/jcph.1996.5639}
}

@article{BeylkinMonzon2002,
  author = {Gregory Beylkin and Lucas Monz{\'o}n},
  title = {On Generalized {Gaussian} Quadratures for Exponentials and Their Applications},
  journal = {Applied and Computational Harmonic Analysis},
  year = {2002},
  volume = {12},
  number = {3},
  pages = {332--373},
  doi = {10.1006/acha.2002.0351}
}

@article{BeylkinMohlenkamp2005SIAM,
  author = {Gregory Beylkin and Martin J. Mohlenkamp},
  title = {Algorithms for Numerical Analysis in High Dimensions},
  journal = {SIAM Journal on Scientific Computing},
  year = {2005},
  volume = {26},
  number = {6},
  pages = {2133--2159},
  doi = {10.1137/040604959}
}

@article{Brandt1977,
  author = {Achi Brandt},
  title = {Multi-Level Adaptive Solutions to Boundary-Value Problems},
  journal = {Mathematics of Computation},
  year = {1977},
  volume = {31},
  number = {138},
  pages = {333--390},
  doi = {10.1090/S0025-5718-1977-0431719-X}
}

@article{BergerOliger1984,
  author = {Marsha J. Berger and Joseph Oliger},
  title = {Adaptive Mesh Refinement for Hyperbolic Partial Differential Equations},
  journal = {Journal of Computational Physics},
  year = {1984},
  volume = {53},
  number = {3},
  pages = {484--512},
  doi = {10.1016/0021-9991(84)90073-1}
}

@article{ElmoatazLezorayBougleux2008,
  author  = {Abderrahim Elmoataz and Olivier Lezoray and S{\'e}bastien Bougleux},
  title   = {Nonlocal Discrete Regularization on Weighted Graphs:
             A Framework for Image and Manifold Processing},
  journal = {IEEE Transactions on Image Processing},
  volume  = {17},
  number  = {7},
  pages   = {1047--1060},
  year    = {2008},
  doi     = {10.1109/TIP.2008.924284}
}

\end{document}